\documentclass[12pt]{amsart}
\usepackage[utf8]{inputenc}

\title[Compressed sensing for inverse problems]{Compressed sensing for inverse problems and the sample complexity of the sparse Radon transform}
\author[G.\ S.\ Alberti]{Giovanni S.\ Alberti}
\address{MaLGa Center, Department of Mathematics, University of Genoa, Via Dodecaneso 35, 16146 Genova, Italy.}
\email{giovanni.alberti@unige.it}

\author[A.\ Felisi]{Alessandro Felisi}
\address{MaLGa Center, Department of Mathematics, University of Genoa, Via Dodecaneso 35, 16146 Genova, Italy.}
\email{alessandro.felisi@edu.unige.it}

\author[M.\ Santacesaria]{Matteo Santacesaria}
\address{MaLGa Center, Department of Mathematics, University of Genoa, Via Dodecaneso 35, 16146 Genova, Italy.}
\email{matteo.santacesaria@unige.it}

\author[S.\ I.\ Trapasso]{S.\ Ivan Trapasso}
\address{Department of Mathematical Sciences ``G.\ L.\ Lagrange'', Politecnico di Torino, Corso Duca degli Abruzzi, 24, 10129 Torino, Italy}
\email{salvatore.trapasso@polito.it}

\date{\today}

\usepackage{amsmath, amssymb, amsfonts, amsthm, amsopn, cancel}
\usepackage{graphicx,tikz}
\usepackage[all]{xy}
\usepackage{amsmath}
\usepackage{bbm}
\usepackage{multirow, longtable, makecell, caption, array, enumitem}
\usepackage{amssymb}
\usepackage{mathtools}
\mathtoolsset{showonlyrefs}
\usepackage{bookmark}
\usepackage{hyperref}

\allowdisplaybreaks

\calclayout

\theoremstyle{plain}
\newtheorem{theorem}{Theorem}[section]
\newtheorem{lemma}[theorem]{Lemma}
\newtheorem{corollary}[theorem]{Corollary}
\newtheorem{proposition}[theorem]{Proposition}

\theoremstyle{definition}
\newtheorem{assumption}[theorem]{Assumption}
\theoremstyle{definition}
\newtheorem{definition}[theorem]{Definition}
\theoremstyle{definition}

\theoremstyle{plain}
\newtheorem*{theorem*}{Theorem}
\theoremstyle{definition}
\newtheorem{example}[theorem]{Example}
\theoremstyle{remark}
\newtheorem{remark}[theorem]{Remark}

\def\bC{\mathbb{C}}
\def\bE{\mathbb{E}}

\def\bN{\mathbb{N}}
\def\bP{\mathbb{P}}
\def\bR{\mathbb{R}}
\def\bS{\mathbb{S}}
\def\bT{\mathbb{T}}
\def\bZ{\mathbb{Z}}

\def\cB{\mathcal{B}}
\def\cD{\mathcal{D}}
\def\cF{\mathcal{F}}
\def\cH{\mathcal{H}}
\def\cHd{L_\mu^2(\cD;\cHt)}
\def\cHt{\mathcal{H}_2}

\def\cL{\mathcal{L}}

\def\cN{\mathcal{N}}
\def\cP{\mathcal{P}}
\def\cS{\mathcal{S}}

\def\cX{\mathcal{X}}
\def\cR{\mathcal{R}}

\def\rd{\bR^d}

\def\lc{\left(}
\def\rc{\right)}

\def\supp{\operatorname{supp}}

\def\*b{*_{\bullet}}

\def\Bd'{B_{\delta'}}

\def\cBd'{\bar{B}_{\delta'}}

\def\Ga{G_\alpha}

\newcommand{\diag}{\mathrm{diag}}
\newcommand{\diff}{\mathrm{d}}
\newcommand{\fc}{\mathfrak{c}}
\newcommand{\Span}{\mathrm{span}}

\begin{document}
\begin{abstract}
Compressed sensing allows for the recovery of sparse signals from  few measurements, whose number is proportional to the sparsity of the unknown signal, up to logarithmic factors. The classical theory typically considers either random linear measurements or subsampled isometries and has found many applications, including accelerated magnetic resonance imaging, which is modeled by the subsampled Fourier transform. In this work, we develop a general theory of infinite-dimensional compressed sensing for abstract inverse problems, possibly ill-posed, involving an arbitrary forward operator. This is achieved by considering a generalized restricted isometry property, and a quasi-diagonalization property of the forward map.

As a notable application, for the first time, we obtain  rigorous recovery estimates for the sparse  Radon transform (i.e., with a finite number of angles $\theta_1,\dots,\theta_m$), which models computed tomography, in both the parallel-beam and the fan-beam settings. In the case when the unknown signal is $s$-sparse with respect to an orthonormal basis of compactly supported wavelets, we prove stable recovery under the condition
\[
m\gtrsim s,
\]
up to logarithmic factors.
\end{abstract}



\date{\today}
\thanks{
This material is based upon work supported by the Air Force Office of Scientific Research under award number FA8655-20-1-7027. We acknowledge the support of Fondazione Compagnia di San Paolo. Co-funded by the European Union (ERC, SAMPDE, 101041040). Views and opinions expressed are however those of the authors only and do not necessarily reflect those of the European Union or the European Research Council. Neither the European Union nor the granting authority can be held responsible for them. The authors are members of the ``Gruppo Nazionale per l’Analisi Matematica, la Probabilità e le loro Applicazioni'', of the ``Istituto Nazionale di Alta Matematica''.
}

\maketitle

\pagebreak
\tableofcontents

\section{Introduction}\label{sec-intro}

\subsection{State of the art}

Compressed sensing (CS) \cite{CRT,donoho2006compressed,eldar-kutyniok-2012,FR,fornasier-2015} has been a very active research area within applied mathematics over the last two decades, characterized by a fruitful exchange between theoretical analysis and technological advances. The focus of CS is the recovery of sparse signals, namely those that can be expressed as a linear combination of a small number of vectors from a reference dictionary. Leveraging sparsity allows for obtaining stable recovery results with a smaller number of measurements than in standard sampling theory.

Over the years, in CS theory there has been a transition from the finite-dimensional setting to the infinite-dimensional one \cite{adcock2016generalized,AHC,AAH,alberti2017infinite,adcock_book}. This shift was largely stimulated by the need to encompass increasingly challenging sensing tasks, which happen to be more faithfully framed in an analog world in terms of suitable operators and function spaces. In fact, the analysis of problems where one aims at recovering a physical quantity given some indirect measurements is the main focus of the theory of inverse problems (IP), with a long tradition in mathematical and numerical analysis \cite{2011-kirsch,mueller2012,scherzer_book,2017-isakov}. Most IP are described by integral or differential equations, and are therefore naturally modeled in an infinite-dimensional setting.

The reconstruction guarantees currently available in the mathematical literature on IP rely on the assumption that an observer has access to an arbitrarily large number of measurements, except for a limited number of special cases (see \cite{alberti2018,Harrach_2019,alberti2022infinite,alberti-arroyo-santacesaria-2023} and references therein, in which, however, no sparsity information is exploited, but only the finite-dimensionality of the unknown). It is clear that in real-life scenarios one can only take a finite number of measurements, usually corrupted by noise and also suffering from limited control on the procedure due to experimental constraints. This discussion calls for the development of a \textit{sample complexity} theory of ill-posed IP, namely, a study of the minimum (finite) number of measurements needed to achieve stable and accurate reconstruction.

It is intuitively clear that the principles of CS could play a role here. Indeed, if the unknown quantity is sparse in some dictionary, then one could approximately locate the signal in some finite-dimensional subspace of the domain of the measurement operator. In other words, a sparsity condition comes along with an inherently finite-dimensional setup. Not surprisingly, such a reduction step is quite delicate and may eventually destroy important information on the genuine infinite-dimensional structure. The framework of \textit{generalized sampling} \cite{ADCOCK2014187,adcock2016generalized,AH1,AHC} addresses this issue, as it relies on finding appropriate sampling and sparsity truncation bandwidths in order to perform the finite-dimensional reduction in a stable way.

Although ideas from CS have influenced the analysis of IP (e.g., sparse regularization \cite{benning_burger,grasmair,grasmair_penalty}), the state of the art is still far from being completely satisfactory, mostly due to technical limitations and a case-by-case approach aimed at leveraging the peculiar features of each model. Despite a large body of literature on the numerical analysis of sparsity-promoting approaches for inverse problems (see, e.g., \cite{2012-jin-maass,jorgensen-sidky-2015,garde2016}), rigorous theoretical guarantees are missing (see \cite{herrolz_ip,alberti2017infinite,campodonico-2021} for some results in this direction).

\subsection{Main contributions}
\subsubsection{Summary}
The purpose of the present article is to move the first steps towards a unified framework for the study of the sample complexity of ill-posed IP under realistic sparsity constraints, by extending the current theory of CS. The main goal is to find the connection between the number of samples needed to achieve recovery of a signal and its sparsity with respect to a suitable dictionary.

In addition to the formulation of the general framework for studying the sample complexity, we present an application to the sparse angle tomography problem. Basically, the goal is to obtain information on the inner structure of an object from the knowledge of a limited number of projections obtained via the Radon transform. Tomographic techniques appear in a variety of real-life situations, including nondestructive testing, biomedical and geophysical imaging, and analysis of astronomical signals. The analysis of tomography IP is a well-established topic and still a very active research area \cite{natterer,natterer_xray,scherzer_g_book,scherzer_book,2021-hansen-etal}. 

While the sparse Radon transform has been thoroughly studied from the numerical point of view and ideas of CS have found applications to tomographic imaging shortly after its appearance \cite{S-Siltanen_2003,2012-arxiv-hansen,sparse-tomography-2013,jorgensen-sidky-2015,Jorgensen_2017}, rigorous results on the sample complexity for the Radon transform and the related technical challenges have been beyond the reach of the CS theory so far. Our main contribution in this respect is to show that signals supported in the unit ball of the plane, assumed to be compressible in a suitable wavelet dictionary, can be reconstructed with high probability if a sufficiently large number of samples of the corresponding Radon transform at different angles is available -- see Theorems \ref{thm:radon-main-thm}, \ref{thm:radon-dir-prob} and \ref{thm:radon-exp-est} below. We consider both the parallel-beam and the fan-beam settings. In simplified terms, we obtain the estimate
\[
\text{number of angles $\gtrsim$ sparsity,}
\]
up to logarithmic factors.

Let us briefly illustrate the key aspects of both the abstract sample complexity framework and its application to the sparse Radon inversion.

\subsubsection{Sample complexity of sparse ill-posed IP}

Let $\cH_1,\cH_2$ be separable complex Hilbert spaces. We assume that the measurements are represented by a family of uniformly bounded linear maps $(F_t)_{t \in \cD} \colon \cH_1 \to \cHt$, indexed over a measure space $(\cD,\mu)$ endowed with a probability distribution $\nu$ such that $\diff{\nu} = f_\nu \diff{\mu}$ for a density function $f_\nu \in L^1(\mu)$, $\|f_\nu\|_{L^1} =1$, satisfying $f_\nu \ge c_\nu$ for some $0<c_\nu \le1$. The forward map $F \colon \cH_1\to L_\mu^2(\cD;\cHt)$ of the model is thus naturally defined by the relationship $F_t u = Fu(t)$, for almost every $t \in \cD$. This setting encompasses both scalar measurements given by inner products with respect to a certain dictionary and by pointwise evaluations. Moreover, we consider measurements taking values in a Hilbert space; this is related to the use of block-sampling strategies in CS \cite{2015-polak-etal,2016-bigot-boyer-weiss}.

We fix once for all an orthonormal basis $(\phi_i)_{i \in \Gamma}$ of $\cH_1$ as a reference dictionary, and let $\Phi\colon \cH_1 \to \ell^2(\Gamma)$ be the corresponding analysis operator -- here $\Gamma$ is an (at most) countable index set, e.g., $\Gamma=\bN$ or $\Gamma\subset\bN\times\bZ$. The signals we consider will be sparse with respect to $(\phi_i)_{i}$.

The dictionary, the measurement operators and the probability density are assumed to be intertwined by a \textit{coherence bound} of the form \[ \|F_t\phi_i\|_{\cHt} \leq B\sqrt{f_{\nu}(t)},\quad t\in\cD, \, i \in \Gamma, \] for some $B\ge1$. 

Consider now an unknown signal $u^\dagger \in \cH_1$. We assume that there exists a finite subset of indices $\Lambda \subset \Gamma$, with $|\Lambda|=M$, such that $\|P_\Lambda^\perp x^\dagger\|$ is small, where we set $x^\dagger \coloneqq \Phi u^\dagger$ and $P_\Lambda$ denotes the orthogonal projection on $\mathrm{span}\{e_i : i \in \Lambda\}$, $(e_i)_{i \in \Gamma}$ being the canonical basis of $\ell^2(\Gamma)$.
Therefore, $\Lambda$ should be thought of as a rough estimate of the support of $x^\dagger$.   One could consider more general forms of tail decay rates for $\|P_{\Lambda}^{\perp} x^\dagger \|_2$, which are usually linked to the regularity of $u^\dagger$ and $\Phi$. We will exploit this connection in Sections~\ref{sec-main res} and \ref{sec-radon} in order to discuss applications to relevant signal classes. For simplicity, here we assume $P_\Lambda^\perp x^\dagger=0$.

Given $m$ i.i.d.\ samples $t_1, \ldots, t_m \in \cD$ drawn from $\nu$, the corresponding measurements are
\[
y_k=F_{t_k}u^\dagger+\varepsilon_k,\qquad k=1,\dots,m,
\]
where the noise vector $\varepsilon\in\cHt^m$ is such that $\|\varepsilon_k\|_{\cHt}\leq\beta$ for some $\beta>0$, for every $k$ -- in other words, we have that the measurement indexed by $t_k$ comes with a noise $\varepsilon_{k}$ that is uniformly bounded over the samples.

 We attempt to reconstruct $x^\dagger$ (and, consequently, $u^\dagger=\Phi^* x^\dagger$) via the $\ell^1$-minimization problem
\begin{equation}\label{eq-minpbm-intro}  
\min_{x\in\ell^2(\Lambda)} \|x\|_{1} \quad : \quad \frac1m\sum_{k=1}^m\|F_{t_k}\Phi^* x-y_k\|_{\cHt}^2 \leq \beta^2, 
\end{equation}

The transition to a finite-dimensional setup is performed via a stable truncation of the forward map $F$ of the model. To be more precise, we consider the $M\times M$ matrix given by
\[ G = \sqrt{P_\Lambda \Phi F^*F \Phi^*  P_\Lambda^*}. \]
Stability ultimately reduces to the invertibility of $G$, which can be ensured if $F$ satisfies a restricted injectivity assumption known as the FBI property, which is usually met in many IP of interest -- for example whenever the operator $F$ is injective, see Definition \ref{def-fbi} below for further details. More generally, this condition can be relaxed if a form of elastic net regularization is introduced in the minimization problem \eqref{eq-minpbm-intro}, so that $G$ can be replaced by $\Ga = \sqrt{P_\Lambda \Phi F^*F \Phi^* P_\Lambda^* + \alpha^2 I_\Lambda}$ for $\alpha > 0$, so that the invertibility of $\Ga$ is no longer a concern. For simplicity, this generalization will not be treated in this work.

A largely simplified version of our abstract result for the sample complexity of abstract inverse problems (Theorem \ref{thm-main-result}) reads as follows.

\begin{theorem*} Under the assumptions outlined so far, there exist universal constants $C_0,C_1,C_2>0$ such that, for any integer $2\le s \le |\Lambda|$, if $\widehat{x}\in\ell^2(\Lambda)$ is a solution of the minimization problem \eqref{eq-minpbm-intro} and the number of samples satisfies
\[ m \geq C_0 B^2\|G^{-1}\|^4 s,\]
up to logarithmic factors, then the following recovery estimate holds with overwhelming probability:
    \[\|x^{\dagger}-\widehat{x}\|_2 \leq C_1\frac{\sigma_s(x^\dagger)_{1}}{\sqrt{s}} + C_2c_{\nu}^{-1/2} \|G^{-1}\| \beta,\]
the term $\sigma_s(x^\dagger)_{1} = \inf\{\|x^\dagger - z \|_1 : z \in \ell^2(\Gamma), |\supp(z)| \le s\}$ representing the error of best $s$-sparse approximation to $x^\dagger$.
\end{theorem*}

Most of the subsequent efforts are directed towards keeping the quantity $\| G^{-1}\|$ under control, in light of its role in the previous result. In Section \ref{subsec-exp-est} we discuss several strategies to provide explicit bounds for $\|G^{-1}\|$, ultimately in connection with the singular values of $F\Phi^* P_\Lambda^*$. In the case where the singular values are not accessible, we introduce a workaround relying on a compatibility condition that is typically met in practice, namely that the measurement operator of the IP approximately acts as a diagonal operator on $(\phi_i)_{i\in\Gamma}$, similarly to the wavelet-vaguelette decomposition \cite{Do}. Roughly speaking, the \textit{quasi-diagonalization} condition introduced in Definition \ref{def-qd} captures how the forward map affects the sparsity of the signal, hence it allows us to obtain bounds for $\|G^{-1}\|$ by mimicking a singular value decomposition.

Moreover, in the spirit of statistical inverse learning \cite{blanchard,bubba-burger-helin-ratti-2021,bubba-ratti-2022} (see Section \ref{sec-relres} below for a brief account), it is possible to further optimize the recovery estimates to make them dependent only on the parameters that can be tuned by the observer, namely the noise level $\beta$ and the number of samples $m$. More precisely, under suitable conditions on the linear and nonlinear approximation errors of $u^\dagger$ with respect to the dictionary $(\Phi_i)_i$ (typically met if $u^\dagger$ belongs to certain classes of signals, such as cartoon-like images -- see the end of $\S$\ref{subsec:sparse-radon} for more details), provided that $m$ is suitably chosen as a function of the noise, we obtain estimates of the form
\[
\|x^{\dagger}-\widehat{x}\|_2 \leq c\beta^\alpha,\qquad \|x^{\dagger}-\widehat{x}\|_2 \leq c\frac{1}{m^{\alpha'}},
\]
up to logarithmic factors, for a suitable constant $c>0$. Here, $\alpha,\alpha'>0$ depend explicitly on the linear and nonlinear approximation errorrates and on the degree of ill-posedness of the forward map $F$. The above estimates hold with high probability, but it would be possible to derive similar estimates for the error in expectation.

We address the reader to Section \ref{sec-main res} for further details and comments on these results, as well as to Sections \ref{sec-rnsp}, \ref{sec-grip} and \ref{sec-sampl} for a detailed exposition of how the standard techniques of CS (in particular, the restricted isometry property and its consequences) must be generalized in order to accommodate the difficulties arising in ill-posed IP. 

\subsubsection{The sparse Radon transform}
\label{subsec:sparse-radon}
In Section~\ref{sec-radon} we discuss how recovery results from finite samples can be obtained for the inversion of the sparse angle Radon transform. Precisely, we consider the reconstruction of a signal supported in the unit ball in $\bR^2$ from samples of the Radon transform along angles $\theta\in[0,2\pi)$.

Recall that the Radon transform along the angle $\theta\in[0,2\pi)$ of a square integrable signal $u$ with compact support in the plane is defined by 
    \[ \cR_\theta{u}(s) \coloneqq \int_{e_\theta^{\perp}} u(y+s e_\theta) \diff{y}, \quad s \in \bR, \]
$\diff{y}$ being the 1D Lebesgue measure on the slice $\theta^{\perp}$. The full Radon transform $\cR u \in L^2([0,2\pi)\times\bR)$ is defined by $\cR{u}(s,\theta) = \cR_\theta{u}(s)$.

Let us describe the dictionary with respect to which the unknown signal $u^{\dagger}\in L^2(\cB_1)$ is assumed to be sparse/compressible. We consider an orthonormal basis of suitable 2D wavelets $(\phi_{j,n})_{(j,n)\in\Gamma}$ with compact support. We assume that there exists $\Lambda \subset \Gamma$ of size $|\Lambda|=M$ such that $\supp x^\dagger \subset \Lambda$ (recall that $x^\dagger=\Phi u^\dagger$) -- hence $\|P_\Lambda^\perp x^\dagger\|_2=0$. In particular, it is natural to assume $\Lambda = \Lambda_{j_0} =\{(j,n) \in \Gamma : j \le j_0\}$ for some $j_0 \in \bN$, namely $u^\dagger$ belongs to a subspace of signals with finite, finest scale $j_0$ in the multiscale decomposition associated with the wavelet basis.

Our goal is to provide recovery estimates for $u^\dagger$ given $m$ samples of the corresponding Radon transform at different angles $\theta_1,\dots,\theta_m\in\bS^1$, drawn i.i.d.\ uniformly at random from $\bS^1$. We consider here only the parallel-beam setting, the fan-beam sampling pattern is discussed in Section~\ref{sec-radon}.

Setting $W=\diag (2^{j/2})_{j,n}$, the nonlinear recovery of $u$ comes through the solution of the minimization problem 
\begin{equation}\label{eq-minpbm-radon-intro}
          \min_u \|W^{-1}\Phi{u}\|_1\quad\colon\quad \frac1m\sum_{k=1}^m\|\cR_{\theta_k}u-y_k\|_{L^2(\bR)}^2 \leq
        \beta^2, 
\end{equation} 
where the minimum is taken over $\Span(\phi_{j,n})_{(j,n)\in\Lambda_{j_0}}$, the sparse-angle measurements are given by $y_k=\cR_{\theta_k}u^\dagger+\varepsilon_k$, and the noise $\varepsilon\in \lc L^2(\bR)\rc^m$ satisfies the uniform bound $\|\varepsilon_k\|_{ L^2(\bR)} \leq \beta$ for every $k$. 

A simplified form of our main result (Theorem~\ref{thm:radon-dir-prob}) on the sparse Radon reconstruction  reads as follows. 

\begin{theorem*} Under the assumptions outlined above, there exist constants $C_0,C_1,C_2>0$ depending only on the wavelet basis such that, for any $2\le s \le  |\Lambda_{j_0}|$, if $\widehat{u}\in\Span(\phi_{j,n})_{(j,n)\in\Lambda_{j_0}}$ is a solution of the minimization problem \eqref{eq-minpbm-radon-intro} and the number of samples satisfies 
\[ m \geq C_0 s,\] 
up to logarithmic factors, then the following recovery estimate holds with overwhelming probability:
    \begin{equation}\label{eq-recbound-radon-intro} \|u^{\dagger}-\widehat{u}\|_{L^2} \leq
        C_1 2^{j_0/2}\frac{\sigma_s(W^{-1}\Phi{u^{\dagger}})_1}{\sqrt{s}} + C_22^{j_0/2}\beta.  \end{equation}
\end{theorem*}
Let us emphasize that the use of weights allows us to obtain a recovery estimate with a smaller number of samples than the ones required when the abstract result above is applied to the Radon setting -- see Theorem \ref{thm:radon-main-thm}.

Finally, we wish to provide a brief exposition of the result proved in Theorem~\ref{thm:radon-exp-est}. If we relax the assumption $\|P_{\Lambda_{j_0}}^\perp x^\dagger \|_2=0$ to a tail bound $\| P_{\Lambda_{j_0}}^\perp x^\dagger\|_2 \le C2^{-aj_0}$ for some $a,C>0$, and the signal $u^\dagger$ satisfies the compressibility condition $\sigma_s(\Phi{u^{\dagger}})_1 \leq C s^{1/2-p}$  for some $p>0$, then the recovery estimates reduce to 
\[  \| \widehat{u} - u^\dagger \|_{L^2} \leq c \beta^{\frac{2a}{2a+1}}, \qquad  \|u^{\dagger}-\widehat{u}\|_{L^2} \leq c \lc\frac{ 1 }{m}\rc^{ \frac{ap}{a+p}}, \]
up to logarithmic factors. These estimates hold under the assumption that the noise is deterministic. Assuming the noise to be a random variable, under natural assumptions (e.g., see \cite{blanchard}) it would be possible to obtain more refined estimates in expectation.

This result is particularly useful for signals from classes with known regularity with respect to a multiresolution dictionary. For instance, it is known that such sparsity constraints in wavelet space with $a=p=1/2$ corresponds to the class of \textit{cartoon-like images}, consisting of 2D piecewise smooth functions apart from mild discontinuities along curves \cite{donoho_cartoon,kuty_cartoon,Ma}. As a result, the reconstruction of a cartoon-like image $u^\dagger$ from $m\asymp \beta^{-2}$ (see Theorem~\ref{thm:radon-exp-est}) noisy samples of the Radon transform at different angles via sparse minimization as above is successful with an error $\| \widehat{u}-u^\dagger\|_{L^2} \leq C\beta^{1/2}$.

\subsection{Related research} 
\label{sec-relres}

Let us briefly discuss other works in the literature that share some common features with the present contribution.

Several approaches to obtain recovery guarantees for sampling and reconstruction in infinite-dimension can be found in the literature. The framework of \textit{generalized sampling} was developed in \cite{AHP,AHP2} to obtain guarantees in separable Hilbert spaces. This setting allows one to recover an element of a finite-dimensional subspace given finitely many inner products with the elements of an arbitrary frame. The issues related to the stability of the reconstruction were addressed by introducing the \textit{balancing property} \cite{adcock2016generalized}, which is well suited in the case of measurements modelled by a unitary operator. Several works have investigated this issue for a variety of nonlinear problems in different scenarios -- we mention \cite{alberti2018,Harrach_2019,alberti2022infinite,alberti-arroyo-santacesaria-2023} in this regard.

One of the main results in this paper (Theorem~\ref{thm-main-result}) generalizes several previously-known coherence-based CS results in finite \cite{candes2007sparsity} and infinite dimension \cite{herrolz_ip,adcock2016generalized,AHC,alberti2017infinite,AAH} in many respects. Indeed, our theory is able to deal with linear ill-posed problems, generally involving compact operators with anisotropic and possibly vector-valued measurements. Since our framework is based on a very weak definition of the sampling/measurement map, which is able to encompass at once both scalar products and pointwise evaluations, this result also intersects the scope of approximation theory, especially in connection with sparse approximation with polynomials \cite{RW} and sparse high-dimensional function approximation \cite{adcock_sparsebook,BDJR,2019-dexter-tran-webster,adcock2022efficient}.

The main notion that allows us to handle the ill-posedness of the forward map and the geometry of the measurements is the \textit{generalized restricted isometry property} (g-RIP) (see Section~\ref{sec-grip}), an extension of the classical \textit{restricted isometry property} (RIP), which has been widely studied in CS due to its fundamental relevance -- see for instance \cite{CT, CRT, CanRIP}. A number of variations on the theme can be found in the literature. In particular, we mention the D-RIP \cite{CYNR,KNW}, which takes into account possible redundancy in the dictionaries; the RIPL \cite{BasHan}, which is suitable for signals with sparsity in levels; and the G-RIPL \cite{AAH}, which is suited to deal with infinite-dimensional signals. The g-RIP considered in this work corresponds to the G-RIPL without levels, and with a weighted sparsity. The role of the matrix $G$ is slightly different in this work, where the main objective is to take into account the ill-posedness of the inverse problem. Alternative approaches to the RIP in CS have been proposed in the literature -- see for instance \cite{BRT, CP}. In \cite{grasmair} an interesting comparison is carried out between the RIP and the weaker notion of \textit{source conditions}; however, it might be difficult to verify if the latter are satisfied in practical scenarios.

The g-RIP is employed here to exploit an approximate diagonalizability condition that is met in many cases of interest in applications -- see Definition \ref{def-qd} of \textit{quasi-diagonalizability}. In fact, similar approaches have already appeared in the literature, a prominent example being the wavelet-vaguelette decomposition (WVD) introduced in \cite{Do} and further developed in \cite{CD-curv-radon, CGL} for more general dictionaries. Let us also point out that in \cite{herrolz_ip} the authors exploit diagonalization properties related to a WVD of the forward map in order to obtain recovery estimates, although the setting and the results differ in several respects from our approach. Moreover, the exact diagonalizability assumption of \cite{herrolz_ip} is too strict for most cases of interest, especially for the sparse Radon transform discussed in this work.

We were able to derive refined estimates in Theorem \ref{thm:interpolation} by leveraging a notion of coherence across scales that is tailored to the multiscale structure of the sparsifying dictionary $(\phi_{j,n})_{(j,n) \in \Gamma}$ -- see Assumption \ref{ass-coherence-bound-D}. Alternative approaches that use multilevel information on coherence are available in the literature. For instance, nonuniform coherence patterns were already observed in some works on MRI \cite{Lustig-et-al, LDP, LDSP}, and have been considered together with multilevel sampling schemes \cite{TD, AAH, AHC, LA}. Our approach considers a modified notion of coherence -- see Remark \ref{rem:rel-coherence} -- that extends these approaches by taking into account the presence of ill-posedness quantified by a certain decay of the singular values of the forward map.

The main application of our abstract framework -- the inversion of the sparse Radon transform -- is a well-known challenge for the IP community. A sparsity-promoting method with respect to wavelets based on a Besov-norm penalty has been considered in \cite{sparse-tomography-2013}, where numerical results are also presented, showing substantial improvement with respect to more classical approaches such as filtered back-projection (FBP) and Tikhonov regularization. Other wavelet-like dictionaries (e.g., shearlets) have been used in the literature to tackle the problem of reconstructing a signal in tomography from sparse data -- see \cite{BSLMP, PKKRNSKS, Bubba_2020}. The reconstruction in \cite{dossal-duval-poon-2017} is performed via sampling in the frequency domain along lines of a signal that is given by a linear combination of Dirac deltas. Under suitable assumptions, it is proved that the number of samples required for exact reconstruction is proportional to the number of deltas. In \cite{jorgensen-sidky-2015, Jorg2}, a numerical study is pursued to investigate how many samples are needed for accurate reconstruction in terms of the sparsity of tomographic images, where sparsity is meant here as the number of nonzero pixels. The numerical experiments presented show that the number of scalar samples required for accurate reconstruction is proportional to the sparsity of the signal.

Finally, we mention an alternative approach to derive recovery guarantees for IP that are similar in spirit  to those presented in this paper, which is given by statistical inverse learning theory, a field that lies at the interface  between IP and statistical learning. In a nutshell, the formulation of the  problem goes as follows. Let $\cH_1$ and $\cH_2$ be Hilbert spaces, $\cH_2$ consisting of functions defined on some input space $(\cX,\nu)$, where $\nu$ is a probability distribution on $\cX$. Let $F\colon\cH_1\rightarrow\cH_2$ be a linear operator. Given $u\in\cH_1$, if $X_1,\dots,X_m\stackrel{\text{i.i.d.}}{\sim}\nu$, consider
\begin{align*}
    Y_k = Fu(X_k)+\varepsilon_k,\qquad k=1,\dots,m,
\end{align*}
where $\varepsilon_k$ are independent noise random variables. The goal is to recover $u$ from the observed values $Y_1,\dots,Y_m$. The case $F=\mathrm{Id}$ reduces to the classical nonparametric regression problem. Reconstruction estimates in this scenario are obtained, for instance, in \cite{blanchard,rastogi,RBM,2020-giordano-nickl}.

\subsection{Discussion and open questions}

Let us discuss here some limitations of our approach and potential directions for future research.
\begin{itemize}
    \item The methods presented in this paper, as well as most of the literature on CS (for some exceptions on specific models, see \cite{ohlsson-2012,blumensath-2013,2017-CS-Generative,gilbertNonlinearIterativeHard2020,berk-etal-2022}), are suited to treat only linear problems, and so are applicable only to linear inverse problems. Nonlinear inverse problems appear in many domains, especially in parameter identification problems for partial differential equations \cite{2011-kirsch,2017-isakov}. Notable examples are the Calderón problem for electrical impedance tomography \cite{2002-borcea,2009-uhlmann} and the inverse scattering problem \cite{2013-colton-kress,MR3601119}. It would be interesting to develop a theory of nonlinear CS to handle these nonlinear problems.
    
    \item We consider a fixed deterministic noise $\varepsilon\in \cH_2^m$, representing a worst case scenario. A common model used in statistical inverse problems \cite{kaipio-somersalo-2007,stuart-2010} and in statistical inverse learning theory \cite{blanchard,2020-giordano-nickl} considers statistical noise, namely, each $\varepsilon_k$ is sampled i.i.d.\ following a probability distribution on $\cH_2$. Understanding how this different setting would affect our method and our reconstruction estimates is an intriguing issue. In particular, since our approach allows us to use variable-density sampling strategies, it would be interesting to study optimality rates for a class of probability distributions.

    \item In this work, we assume that the unknown signal $u^\dagger$ is sparse with respect to an orthonormal basis $(\phi_i)_{i \in \Gamma}$ of $\cH_1$. We expect that the whole machinery would work also under relaxed assumptions, for instance, if $(\phi_i)_{i \in \Gamma}$ is a frame and/or a Riesz basis, as in  \cite{poon-2017,alberti2017infinite,BDJR}. Such generalizations would allow us to consider dictionaries that are more general than wavelet orthonormal bases, such as curvelets or shearlets, which are known to provide better rates for the nonlinear approximation error \cite{CD,GL,kuty_cartoon}. As a consequence, this is expected to give better sample complexity estimates. We leave this extension for future research.

     \item In the model of the sparse Radon transform we considered a finite number of angles $\theta_k$ but, for each fixed angle, we sample the Radon transform for every translation $s\in [-1,1]$. As such, this is a semi-discrete model for the measurements, and it is natural to wonder whether the same estimates, or similar ones,  hold true with a sampling also in the  variable $s$. The work \cite{dossal-duval-poon-2017} contains some results in this direction, but with sampling in frequency and and in the setting of Dirac delta spikes.

    \item We have opted to present only one application (the sparse Radon transform) of the abstract sample complexity result for inverse problems. However, the proposed approach is  flexible enough to encompass a number of classical problems, including the recovery of wavelet coefficients of a $L^2$ signal from its samples in the Fourier domain, as well as deconvolution problems with finite samples. To the best of our knowledge, this is the first time that IP so different in nature (namely, involving ``discrete" and ``continuous" samples) can be handled at one time by specializing the same abstract framework.  We decided to postpone this analysis to a separate contribution.
\end{itemize}

\subsection{Structure of the paper} Let us briefly illustrate the organization of the contents.  In Section \ref{sec-sec-setting} we offer a detailed treatment of the preliminary notions that are necessary to frame and state the main results. After such a preparation, in Section \ref{sec-main res} we formulate our general result for the sample complexity of ill-posed IP under sparsity constraints. The main application is the sparse angle tomography problem, which is thoroughly discussed in Section \ref{sec-radon}. All the proofs of both main and ancillary results can be found in Section \ref{sec-proofs}, with the exception of the proof of a key technical result (Theorem~\ref{thm-samp-grip}), which can be found in Appendix~\ref{appendix:proofs}.

\section{Setup}\label{sec-sec-setting}

\subsection{Weighted setup and notation}
\label{sec-weight-set}
Let $\bN$ denote the set of positive integer numbers. For $n \in \bN$, we denote  the set of the positive integers up to $n$ by $[n]$, namely $[n]=\{1, \ldots, n\}$.

Let $\Gamma$ be a finite or countable index set, e.g., $\Gamma=\bN$ or $\Gamma\subseteq\bN\times\bZ$. In our setting, it is the index set of the dictionary representing the unknown signals. We will consider also a finite subset of indices $\Lambda \subseteq \Gamma$, which stands for the index set where the signal reconstruction is carried out. We denote by $P_\Lambda$ the orthogonal projection on $\ell^2(\Gamma)$ defined by \[ (P_\Lambda x)_i = \begin{cases} x_i & (i \in \Lambda) \\ 0 & (i \notin \Lambda) \end{cases}. \] The image of $P_\Lambda$ is thus $\ell^2_\Lambda(\Gamma) \coloneqq \mathrm{span} \{ e_i : i \in \Lambda \}$, $(e_i)_{i \in \Gamma}$ being the  canonical basis of $\ell^2(\Gamma)$. With an abuse of notation, we often identify $\ell^2_\Lambda(\Gamma)$ with $\ell^2(\Lambda)$ or with $\bC^M$, where $M=|\Lambda|$. We denote the corresponding adjoint map by $\iota_\Lambda$, that is the canonical embedding $\ell^2_\Lambda(\Gamma) \to \ell^2(\Gamma)$. We also set $P_\Lambda^\bot \coloneqq I - P_\Lambda$, where $I$ is the identity operator.

We say that a sequence $\omega\in\bR^{\Gamma}$ is a vector of \textit{weights} if $\omega_i\geq 1$ for every $i\in \Gamma$. The \textit{weighted size} of a set $S\subseteq \Gamma$ is  defined by
\begin{equation}
    \omega(S) \coloneqq \sum_{i\in S} \omega_i^2,
\end{equation} whenever $S$ is finite. Denoting by $|S|$ the cardinality of a set $S$, the following inequalities are immediately verified:
\begin{equation}
    |S| \leq \omega(S) \leq |S|\sup_{i\in S}\omega_i^2.
\end{equation} 

For $0<p \le 2$ we introduce the set 
\[ \ell^p_\omega(\Gamma) \coloneqq \{ x \in \bC^\Gamma : \| x\|_{p,\omega} < \infty\}, \] where the \textit{$\omega$-weighted $\ell^p$} norm is defined by
\begin{equation}
    \|x\|_{p,\omega} \coloneqq \lc \sum_{i\in \Gamma} |x_i|^p \omega_i^{2-p}\rc^{1/p}.
\end{equation} 
The limit case $p\to 0$ leads to the notion of \textit{weighted sparsity}. To be precise, denoting  the support of $x \in \bC^\Gamma$ by $\supp (x) = \{ i \in \Gamma : x_i\ne 0\}$, we say that $x$ is $s$-$\omega$-sparse if
\[ \| x \|_{0,\omega} \coloneqq \omega(\supp x) \le s. \] Note that in the unweighted case (i.e., $\omega_i = 1$ for all $i\in \Gamma$) we obtain the usual $\ell^p$ norm $\|x \|_p$ and the cardinality of the support $\|x\|_0 = |\supp x|$.

We also introduce the \textit{error of best $\omega$-weighted $s$-sparse approximation} of $x \in \ell^2(\Gamma)$ with respect to the weighted $\ell^p$ norm:
\[ \sigma_s(x)_{p,\omega} \coloneqq \inf \{\|x-y\|_{p,\omega} : y \in \bC^{\Gamma},  \|y\|_{0,\omega} \le s\}. \] 
Equivalently, we have that
\begin{equation}
    \sigma_s(x)_{p,\omega} =
    \inf\{ \|x_{S^c}\|_{p,\omega}\colon S\subset \Gamma,\ \omega(S)\leq s \},
\end{equation}
where $x_S \coloneqq P_S x$ and $S^c=\Gamma\setminus S$. If $\tilde{S}\subset \Gamma$ is a set such that $\sigma_s(x)_{p,\omega}=\|x_{\tilde{S}^c}\|_{p,\omega}$ and $\omega(\tilde S)\le s$, we say that $x_{\tilde{S}}$ is a \textit{best s-$\omega$-sparse approximation to $x$} with respect to $\ell_{\omega}^p$. Note that $x_{\tilde{S}}$ is not unique in general.

\subsection{Dictionaries, measurements and the minimization problem}\label{sec-setting}
The Hilbert spaces appearing below are assumed to be complex and separable.

\subsubsection*{Hilbert spaces} Let $\cH_1$ and $\cHt$ be Hilbert spaces. We will refer to $\cH_1$ as the space of signals and to $\cHt$ as the space of measurements.

\subsubsection*{Dictionary} Let $(\phi_i)_{i\in\Gamma}$ be an orthonormal basis of $\cH_1$.
Let $\Phi\colon\cH_1\rightarrow\ell^2(\Gamma)$ be the corresponding analysis operator, namely $\Phi u\coloneqq (\langle u,\phi_i \rangle_{\cH_1})_{i\in\Gamma}$. The corresponding synthesis operator $\Phi^*\colon\ell^2(\Gamma)\rightarrow\cH_1$ is given by $\Phi^*x=\sum_{i\in\Gamma} x_i\phi_i$.

\subsubsection*{Measurement space} Let $(\cD,\mu)$ be a measure space. We  think of $\cD$ as a space of parameters associated with the  measurements. Indeed, we  perform random sampling with respect to a probability measure $\nu$ on $\cD$ that is absolutely continuous with respect to $\mu$ -- namely, there exists a positive function $f_{\nu}\in L^1(\mu)$ such that $\|f_{\nu}\|_{L^1}=1$, and  $\diff{\nu}\coloneqq f_{\nu}\diff{\mu}$. The choice of the probability distribution turns out to be crucial in order to obtain an optimal sample complexity; however, it will not be very relevant for the purpose of the present work. We will investigate this matter further in future work.

\subsubsection*{Measurement operators and forward map}
Let $(F_t)_{t\in\cD}\colon\cH_1\rightarrow\cHt$ be a family of uniformly bounded linear maps, with $\|F_t\|\leq C_F$ for every $t\in\cD$, modelling the measurement operators in our problem. We suppose that
\begin{equation}
    \cD\ni t \mapsto F_t u\in\cHt
\end{equation}
belongs to the Bochner space $L_{\mu}^2(\cD;\cHt)$ for every $u\in\cH_1$. We define the forward map
\begin{equation}\label{def:forward}
    F\colon \cH_1\to L_\mu^2(\cD;\cHt),\qquad (Fu)(t)=F_tu,
\end{equation}
for almost every $t\in\cD$. We assume that $F$ is bounded with $\|F\|_{\cH_1\to L_\mu^2(\cD;\cHt)}\le C_F$.

The attentive reader will notice that assuming (non-uniform) boundedness of the measurement operators would affect only the estimates involving the truncation error -- see Proposition \ref{prop:truncation-error}. Nevertheless, as the uniform boundedness assumption is satisfied in the context of the applications discussed below, we prefer to trade off some generality for consistency.

\subsubsection*{Noise} 
Letting $u^\dagger \in \cH_1$ denote the unknown signal and if $t_1,\dots,t_m$ are i.i.d.\ samples from the distribution $\nu$ on $\cD$, the noisy measurements may be written as 
\begin{equation}\label{eq:A}
    y_k \coloneqq F_{t_k}u^\dagger +\varepsilon_k, \quad k=1,\dots,m,
\end{equation}
with $\varepsilon_k\in\cHt$. We assume the following uniform bound on the noise:
\begin{equation}\label{eq:noise_bounda}
\max_k \|\varepsilon_k\|_{\cHt}\leq\beta,
\end{equation}
for some $\beta\ge 0$.

\subsubsection*{Truncation error}
Let $\Lambda\subseteq\Gamma$ be finite. We assume that there exists $r\ge 0$ such that $\|P_{\Lambda}^{\perp}x^{\dagger}\|_2\leq r$, where $x^\dagger=\Phi u^\dagger$. The truncation error $\|P_{\Lambda}^{\perp}x^{\dagger}\|_2$ is often referred to as linear approximation error \cite{devore-1998}. We show below (see Example \ref{ex:cartoon-like}) that, if $u^\dagger$ and $\Phi$ are sufficiently regular in a sense to be specified, then the decay rate of  $\|P_{\Lambda}^{\perp}x^\dagger\|_2$ can be explicitly given in terms of $M=|\Lambda|$.

\subsubsection*{Minimization problem}
We consider the following minimization problem:
\begin{equation}
\label{eq-min-probl1}
    \min_{x\in\ell^2(\Lambda)} \|x\|_{1,\omega}\quad :\quad  \frac{1}{m} \sum_{k=1}^m \|F_{t_k}\Phi^*\iota_{\Lambda}x-y_k\|_{\cHt}^2 \leq \eta^2,
\end{equation}
for some $\eta\ge 0$, where we recall that $\iota_{\Lambda}=(P_{\Lambda})^*$. Since the problem involves the minimization of a convex functional under a quadratic constraint, there exists at least one solution -- unless the constraint condition is empty.

\begin{remark}\label{rem-an-vs-syn}
    The $\ell^1$ minimization problem \eqref{eq-min-probl1} is known in the literature as the \textit{synthesis formulation}. An alternative approach is the so-called \textit{analysis formulation}
    \[ 
        \min_{u} \|\Phi{u}\|_1\quad:\quad
        \frac{1}{m}\sum_{k=1}^m\|F_{t_k}u-y_k\|_{\cHt}^2\leq \eta^2,
    \]
    where the minimum is taken over $\Span(\phi_{i})_{i\in\Lambda}$. 
    
    For general dictionaries, the two approaches are not equivalent -- see for instance \cite{EMR}. However, equivalence holds for orthonormal bases, which is the case in the present work. Indeed, $\Phi$ and its inverse $\Phi^*$ are isometric isomorphisms of $\cH_1$ and $\ell^2(\Gamma)$. We adopt the synthesis approach in the statement of theorems of the current section as this formulation is more natural in the context of compressed sensing and the restricted isometry property -- see Section~\ref{sec-grip}. Instead, in Section \ref{sec-radon} we opt for the analysis approach, which best fits the spirit of inverse problems.
\end{remark}

\subsubsection*{Two examples}
We now briefly discuss two examples of problems that fit into our framework: the reconstruction of the wavelet coefficients of a sparse signal from Fourier measurements and the recovery of a sparse function in a basis of polynomials from a finite number of pointwise evaluations. Even though they have been thoroughly studied -- see, for instance, \cite[Section 12.1]{FR} -- our framework generalizes compressed sensing in bounded orthonormal systems in many respects, including the introduction of a possibly ill-posed forward map and the presence of more general vector-valued measurement operators.

\begin{example}[\textbf{Reconstruction of wavelet coefficients from Fourier samples}]
\label{ex-wavelet-fourier}
We consider the problem of reconstructing the wavelet coefficients of a signal supported in $\bT=[0,1]$ by sampling the Fourier coefficients in an appropriate bandwidth $[N]_\pm:=\{-N, -N+1, \ldots, N-1, N\}$; we refer the interested reader to \cite{CRT, LDSP, LDP}.

In this context, we have that $\cH_1= L^2(\bT)$ is the space of signals and $\cHt= \bC$ is the space of scalar measurements. We consider an orthonormal basis $(\phi_i)_{i\in\Gamma}= (\phi_{j,n})_{(j,n)\in\Gamma}$ of boundary (\cite{cohen_wave}) or periodic (\cite[Section~3.11]{Me}) Daubechies wavelets. 
The index set $\Gamma\subset\bN\times\bZ$  here includes all the possible scales, indexed by $j\in\bN$, and all the possible values of the translation parameter, indexed by $n\in\bZ$. In general, the range of $n$ depends on the scale $j$.

As measurement space for this model we consider $(\cD,\mu)=([N]_\pm,\fc)$, where $\fc$ is the counting measure. The forward map $F\colon L^2(\bT)\rightarrow\bC^{2N+1}$ is given by $F= P_N\cF$, where $\cF\colon L^2(\bT)\rightarrow\ell^2(\bZ)$ is the Fourier transform
\[
\cF u(t)=\int_{\bT} u(x)e^{-2\pi i xt }\,\diff x, \qquad t\in \bZ,
\]
and $P_N\colon\ell^2(\bZ)\rightarrow\bC^{2N+1}$ is the projection on the components indexed by $[N]_\pm$. For every $t\in[N]_\pm$, the measurement operators $F_t\colon L^2(\bT)\rightarrow\bC$ are given by $F_t u=\cF{u}(t)=Fu(t)$.  We choose $f_{\nu}(0)\coloneqq 1/C_{\nu}$ and $f_\nu(t) \coloneqq 1/(|t|C_{\nu})$ for $t\neq 0$, where $C_{\nu}\coloneqq1+\sum_{t=1}^N 2/t$ is a normalizing constant; see (i) in Remark \ref{rem-examples-coherence} for further insights on this choice.

The truncation bandwidth $\Lambda\subset\Gamma$ is the subset of indices $(j,n)$ such that $j\leq j_0$. In other words, we approximate the original signal with one having a finite finest scale $j_0$. 
\end{example}

\begin{example}[\textbf{Approximation with polynomials}]
\label{ex-poly}
We consider the classical problem of reconstructing a polynomial of degree at most $M$ that is sparse with respect to an orthonormal basis of polynomials, given a finite number of pointwise samples \cite{RW}.

Let $L^2_w(-1,1)$ be the space of square-integrable functions with respect to a weight $w\colon [-1,1]\to\bR_+$, where $w\diff{x}$ is a probability measure. Let $(p_i)_{i\in\bN}$ be an orthonormal basis of polynomials of $L^2_w(-1,1)$. We fix the dictionary $(\phi_i)_{i \in \Gamma} = (p_i)_{i\in \Gamma}$, $\Gamma = [M+1]$, and define the space of polynomials of degree at most $M$ by $\cP_{M} \coloneqq \Span(p_i)_{i\in\Gamma}\subset L^2_w(-1,1)$. The model under consideration is represented by $\cH_1=\cP_M$  and $\cHt=\bC$. 

The measurement space is given by $(\cD,\mu) = ([-1,1],w \diff{x})$. The forward map coincides with the identity ($F = I$) and, for every $t\in[-1,1]$, the measurement operators $F_t\colon\cP_{M}\rightarrow\bC$ are given by $F_t{u} = u(t)$. We choose $\Lambda=\Gamma$ and $f_\nu\equiv 1$, so that $\nu=\mu$.
\end{example}

\section{A general result on sample complexity}\label{sec-main res}

\subsection{Main result}

In order to state our main result we need to introduce a few assumptions.

\begin{assumption}\label{ass-coherence-bound}
The following coherence bound is satisfied:
\begin{equation}
    \|F_t\phi_i\|_{\cHt} \leq B\sqrt{f_{\nu}(t)}\,\omega_i,\quad t\in\cD,\ i\in\Gamma,
\end{equation}
for some $B\geq 1$, where $f_\nu$ was defined in Section~\ref{sec-setting} and $\omega_i\geq 1$.
\end{assumption}
\begin{remark}
\label{rem-examples-coherence}
Let us discuss the occurrence of this bound in the two examples already discussed.  
\begin{itemize}
    \item[\textit{(i)}] Consider the Fourier-wavelet sensing problem discussed in Example \ref{ex-wavelet-fourier}. It is known from \cite[Theorem 2.1]{JAH} that the following estimates hold under suitable regularity assumptions on the scaling functions of the wavelet dictionary:
\begin{align*}
    |\cF{\phi_{j,n}}(0)| &\leq C,\quad (j,n)\in\Gamma,\\
    |\cF{\phi_{j,n}}(t)| &\leq \frac{C}{\sqrt{|t|}},\quad t\in\bZ\setminus\{0\},\ (j,n)\in\Gamma
\end{align*}
for some $C\geq 1$. Letting $C_{\nu}\coloneqq 1 + \sum_{t=1}^N 2/t$, $f_{\nu}(0)=1/C_{\nu}$ and $f_{\nu}(t)=1/(|t|C_{\nu})$ for $t\neq 0$ and $B\coloneqq C\sqrt{C_{\nu}}$, then the above estimate can be recast in the form of Assumption \ref{ass-coherence-bound} with $\omega_i\equiv 1$.

 \item[\textit{(ii)}] Consider the problem in Example \ref{ex-poly}. In the case of approximation with Legendre polynomials $(p_i)_{i\in\bN}$ on $[-1,1]$ (see \cite{RW} or \cite[pag. 504]{Z}), we have
\begin{equation}
    \| p_i\|_{\infty} \leq \sqrt{2i-1},\quad i\in\bN.
\end{equation}
Assumption \ref{ass-coherence-bound} can thus be satisfied provided that the weights are chosen in accordance with the constraint $\omega_i\geq\sqrt{2i-1}$.
\end{itemize} 
\end{remark}

\begin{assumption}\label{ass-prob-low-est}
The following estimate is satisfied for some $c_{\nu}>0$:
\begin{equation}
    c_{\nu}\leq f_{\nu}\leq 1.
\end{equation}
\end{assumption}
\begin{remark}
The inequality is a rather technical assumption and, as will be clear from the proof of Theorem \ref{thm-main-result}, it is related to the fact that we suppose to have only a uniform bound on the noise level for the samples, see \eqref{eq:noise_bounda}. By strengthening the requirement on the noise, this assumption can be dropped, as will be shown in future work. We do not include this extension here as it is not relevant for the sparse Radon problem treated in this work.
\end{remark}

If $A$ is a matrix or an operator, we let $\|A\|$ denote its operator norm.

\begin{theorem}
\label{thm-main-result}
Consider the setting of Section \ref{sec-setting}. Let Assumptions \ref{ass-coherence-bound} and \ref{ass-prob-low-est} be satisfied. Set $M=|\Lambda|$ and define the following $M\times M$ matrix:
\begin{equation}
    G \coloneqq \sqrt{P_{\Lambda} \Phi F^*F\Phi^*\iota_{\Lambda}}.
\end{equation}
We assume that $G$ is invertible.

Let $x^{\dagger}\in\ell^2(\Gamma)$ be such that $\|P_{\Lambda}^{\perp}x^{\dagger}\|_2\leq r$. Consider $m$ i.i.d.\ samples $t_1,\dots,t_m\in\cD$ drawn from the probability distribution $\nu$. Let $y_k \coloneqq F_{t_k}\Phi^* x^{\dagger}+\varepsilon_k$, for $k=1,\dots,m$, with $\max_k \|\varepsilon_k\|_{\cHt}\leq \beta$.

There exist absolute constants $C_0,C_1,C_2>0$ such that the following holds. Let $\widehat{x}\in\ell^2(\Lambda)$ be a solution of the minimization problem
\begin{equation}
\label{eq-min-probl}
    \min_{x\in\ell^2(\Lambda)} \|x\|_{1,\omega}\quad :\quad \frac{1}{m}\sum_{k=1}^m \|F_{t_k}\Phi^*\iota_{\Lambda} x-y_k\|_{\cHt}^2 \leq \lc \beta + C_3\|G^{-1}\|^{-1}r \rc^2.
\end{equation} 
Fix $s \in [M]$ with $s\geq\|\omega\|_{\infty}^2/4$ and set $\tau\coloneqq B^2\|G^{-1}\|^4\|G\|^2 s$. For any $\gamma\in(0,1)$, if
\begin{equation}
    m \geq C_0\tau\max\{\log^3{\tau}\log{M},\log(1/\gamma)\},
\end{equation}
then, with probability exceeding $1-\gamma$, the following recovery estimate holds:
\begin{equation}
\label{eq-rec-est1}
    \|x^{\dagger}-\widehat{x}\|_{1,\omega} \leq C_1 \sigma_s(P_{\Lambda}x^{\dagger})_{1,\omega} + C_2 \sqrt{s} c_{\nu}^{-1/2} (\|G^{-1}\|\beta + C_3 r) + \|P_{\Lambda}^{\perp} x^{\dagger}\|_{1,\omega} ,
\end{equation}
\begin{equation}
\label{eq-rec-est2}
    \|x^{\dagger}-\widehat{x}\|_2 \leq C_1 \frac{\sigma_s(P_{\Lambda}x^{\dagger})_{1,\omega}}{\sqrt{s}} + C_2 c_{\nu}^{-1/2} (\|G^{-1}\|\beta + C_3 r),
\end{equation}
where $C_3\coloneqq c_{\nu}^{-1/2}C_4( \|F\Phi^*\iota_{\Gamma\setminus\Lambda}\| \|G^{-1}\| + 1)$, and $C_4$ depends only on $C_F$.
\end{theorem}

Theorem \ref{thm-main-result} presents a recovery guarantee of nonuniform type, i.e., valid with overwhelming probability for a fixed unknown signal. Nonetheless, all the tools used in the proof are inherently suited for uniform recovery except for Proposition~\ref{prop:truncation-error} only, which is necessary to obtain refined estimates on the truncation error. It is possible to obtain a uniform recovery result (that is, valid for a whole family of suitably sparse signals) by replacing the term $\|G^{-1}\|\beta+C_3 r$ in the estimates with the possibly worse bound $\|G^{-1}\|(\beta+r)$. In particular, the two terms coincide if $x^{\dagger}$ is supported in $\Lambda$ and therefore $r=0$, leading to uniform recovery estimates for this class of signals. 

The factor $\log^3{\tau}$ appearing in the estimates is known to be suboptimal -- see for instance \cite{BDJR}. Optimizing the exponent of the logarithm would lead to a substantial increase in the length of the proofs and falls outside the scope of the present work.
\bigskip

In order to apply this result to specific cases of interest, several quantities must be explicitly computed, the most crucial one being $\|G^{-1}\|$. This will require new tools and will be the focus of $\S$\ref{subsec-exp-est} below; we anticipate here a few general observations.

A sufficient condition for the invertibility of $G$ that is typically satisfied in several inverse problems of interest is a restricted injectivity property for $F$, as detailed below.

\begin{definition}\label{def-fbi}
Let $\cH$ be a complex and separable Hilbert space. A linear operator $V\colon \ell^2(\Gamma) \to \cH$ is said to possess the \textit{finite basis injectivity (FBI) property} if, for every finite family of indices $\Lambda \subset \Gamma$, the restriction of $V$ to $\ell^2(\Lambda)$ is injective.

Given an orthonormal basis $(\phi_i)_{i\in\Gamma}$ of a Hilbert space $\cH_1$ with analysis operator $\Phi$, an operator $F\colon \cH_1 \to \cH$ is said to possess the \textit{FBI property with respect to $\Phi$} (in short: $\Phi$-FBI) if $F\Phi^*$ has the FBI property.
\end{definition}

Suppose that $F$ is $\Phi$-FBI. Then we have that $F\Phi^*\iota_{\Lambda}$ is injective for every finite $\Lambda\subset\Gamma$, which implies that $G\in\bC^{M\times M}$ is positive definite and thus invertible.

As already observed in \cite{Br}, the FBI condition is quite natural in the context of inverse problems with sparsity constraints -- in fact, it is less restrictive than full injectivity, even if the latter is more common in the IP literature. An elementary example of a non-injective FBI operator on $\ell^2(\bN)$ is the following one:
\begin{equation}
    Ve_i =
    \begin{cases}
        \sum_j c_j e_j & (i=1)\\
        e_{i-1} & (i\geq 2),
    \end{cases}
\end{equation}
where $(e_i)_{i\in\bN}$ is the standard basis of $\ell^2(\bN)$ and $(c_j)_{j\in\bN}$ is any  complex sequence such that $\sum_j |c_j|^2<\infty$ and $c_j\neq 0$ for every $j$.

Another quantity that needs to be explicitly computed in order to apply Theorem~\ref{thm-main-result} is $C_3 =c_{\nu}^{-1/2}C_4( \|F\Phi^*\iota_{\Gamma\setminus\Lambda}\| \|G^{-1}\| + 1)$. We find that under reasonable conditions (see Theorem~\ref{thm:interpolation}), the quantity $C_3$ can be bounded by a constant depending only on $C_F$, on $c_{\nu}$ and on the quasi-diagonalization bounds.

\subsection{Quantitative bounds on \texorpdfstring{$\|G^{-1}\|$}{|G-1|} and quasi-diagonalization}
\label{subsec-exp-est}

We have already shown that if $F$ is $\Phi$-FBI then the matrix $G$ is non-degenerate. In order to make the estimates of Theorem~\ref{thm-main-result} effective, we now derive explicit bounds for $\| G^{-1 }\|$. In the literature on infinite-dimensional compressed sensing, such bounds were obtained with the so-called \textit{balancing property} \cite{adcock2016generalized}, which is well suited in the case of measurements modeled by a unitary operator. Here we somewhat exploit the (possible) ill-posedness of our forward map as well as a \textit{quasi-diagonalization} property, introduced below, that is satisfied by many forward maps of interest, including the Radon transform. We then derive ad hoc bounds that are of a different flavour compared to the ones obtained with the balancing property. Potential combinations of the balancing property with the quasi-diagonalization property are left for future work.

We note that $\|G^{-1}\|$ can be explicitly bounded using the singular values of $F\Phi^*$. Indeed, we have that
\begin{equation}
    \|G^{-1}\| \leq \sigma_M(F\Phi^*\iota_{\Lambda})^{-1},
\end{equation}
where $\sigma_M(F\Phi^*\iota_{\Lambda})$ is the $M$-th singular value of $F\Phi^*\iota_{\Lambda}$, with $M=|\Lambda|$.

\subsubsection*{Quasi-diagonalization}
The previous estimate depends on the singular values of $F\Phi^*\iota_{\Lambda}$. In practical scenarios, estimating $\sigma_M(F\Phi^*\iota_{\Lambda})$ is non-trivial. For this reason, we now discuss an alternative approach for estimating $\|G^{-1}\|$, which is a consequence of a condition that is typically met in several contexts -- more precisely, a condition that holds when it is possible to identify a dictionary that approximately diagonalizes the action of the forward map of the model.

We consider now (orthonormal) multiresolution dictionaries $(\phi_{j,n})_{(j,n)\in\Gamma}$  with two indices (e.g., orthonormal wavelets). The parameter $j$ represents the scale, and we use the convention that higher values of $j$ correspond to finer scales.
\begin{definition}\label{def-qd}
We say that $F$ satisfies the \textit{quasi-diagonalization property} with respect to $(\Phi,b)$, $b\geq 0$, if there exist $c,C>0$ such that
\begin{equation}
\label{eq-quasi-diag-gamma}
    c\sum_{(j,n)\in\Gamma} 2^{-2bj}|x_{j,n}|^2 \leq
    \|F\Phi^* x\|_{\cHd}^2 \leq
    C\sum_{(j,n)\in\Gamma} 2^{-2bj}|x_{j,n}|^2,\quad x\in\ell^2(\Gamma).
\end{equation}
\end{definition}
Roughly speaking, the previous property entails the fact that the action of the forward map on the elements of the dictionary $(\phi_{j,n})_{(j,n) \in \Gamma}$ approximately coincides with that of a diagonal operator with dyadic coefficients depending only on the parameter $b$ and the ``scale" index $j$: as a rule of thumb, the larger $b$ is, the more smoothing $F$ is -- hence, the ill-posedness of the corresponding inverse problem gets worse.

We now provide a set of conditions ensuring that the quasi-diagonalization property is satisfied. Recall that $H^b(\rd)$, $b\in \bR$, is the Sobolev space of distributions $u\in \cS'(\rd)$ such that $(I-\Delta)^{b/2}u \in L^2(\rd)$, while $H^b(\bT^d)$ is the Sobolev space of distributions $u\in\cS'(\bT^d)$ such that $(I-\Delta)^{b/2}u\in L^2(\bT^d)$, where $(I-\Delta)^{b/2}$ corresponds to the Fourier multiplier $(1+|\cdot|^2)^{b/2}$ on $\bR^d$ or on $\bZ^d$, respectively.
\begin{proposition}
\label{prop-cond-quasi-diag}
Let $\cH_1$ be either $L^2(\bT^d)$ or $L^2(\bR^d)$. Assume that there exist $c,C>0$, $b\ge0$ such that:
\begin{itemize}
    \item The forward map $F$ satisfies
    \begin{equation}
    \label{eq-suff-quasi-diag2}
        c\|u\|_{H^{-b}}^2 \leq \|Fu\|_{\cHd}^2 \leq C\|u\|_{H^{-b}}^2,\quad u\in \cH_1.
    \end{equation}
    \item The dictionary $(\phi_{j,n})_{(j,n)\in\Gamma}$ is an orthonormal basis of $\cH_1$ such that
    \begin{equation}
    \label{eq-suff-quasi-diag1}
        c\|u\|_{H^{-b}}^2 \leq \sum_{(j,n)\in\Gamma} 2^{-2bj} |\langle u,\phi_{j,n} \rangle_{L^2}|^2 \leq C\|u\|_{H^{-b}}^2,\quad u\in \cH_1,
    \end{equation}
which we will refer to as the Littlewood-Paley property of the dictionary.
\end{itemize}
Then $F$ satisfies the quasi-diagonalization property \eqref{eq-quasi-diag-gamma} with respect to $(\Phi,b)$.
\end{proposition}
\begin{remark}
\label{rem-regwav-quasidiag}
If $(\phi_{j,n})_{(j,n)\in\Gamma}$ is a $q$-regular wavelet system (in the sense of \cite[Section 2.2]{Me}), then the Littlewood-Paley condition \eqref{eq-suff-quasi-diag1} is satisfied for every $-q<b<q$ -- see for instance \cite[Theorem 8 in Section 2.8 and Section 3.11]{Me} and \cite[Theorem 9.2]{Ma}. Compactly supported $q$-regular wavelet system exist for every fixed value of $q\in\bN$ -- see \cite[Theorem 3 in Section 3.8]{Me}.
\end{remark}

\subsubsection*{Coherence bounds across scales}
The results of Theorem \ref{thm-main-result} can be improved in the quasi-diagonalization regime with respect to $(\Phi,b)$ as soon as additional coherence bounds are available, tailored to the multiscale structure of the dictionary $(\phi_{j,n})_{(j,n) \in \Gamma}$.

In particular, we consider the following tighter version of Assumption \ref{ass-coherence-bound}.

\begin{assumption}\label{ass-coherence-bound-D}
The following coherence bound is satisfied:
\begin{equation}
    \|F_t\phi_{j,n}\|_{\cHt} \leq B\frac{\sqrt{f_{\nu}(t)}}{d_{j,n}}\omega_{j,n},\quad t\in\cD,\ (j,n)\in\Gamma.
\end{equation}
for some $B\geq 1$, $1\leq d_{j,n}\leq 2^{bj}$, where $b$ is the constant appearing in the quasi-diagonalization \eqref{eq-quasi-diag-gamma}, $f_\nu$ was defined in Section~\ref{sec-setting} and $\omega_{j,n}\geq 1$.
\end{assumption} 
\noindent Remark~\ref{rem:rel-coherence} below gives further insights on the role of the parameters $d_{j,n}$.

We can take advantage of such refined coherence information to substantially reduce the number of samples that are needed to achieve recovery in Theorem \ref{thm-main-result}, retaining the same recovery rates.

For $j \in \bN$, we introduce the index subsets with finest scale $j$
\begin{equation}\label{eq:Lambda_j}
    \Lambda_j \coloneqq \{(j',n)\in\Gamma\colon\ j'\leq j\}. 
\end{equation}
Moreover, we always assume that $|\Lambda_j|<+\infty$.

\begin{theorem}
\label{thm:interpolation}
Consider the setting of Section \ref{sec-setting}. Let $(\phi_{j,n})_{(j,n)\in\Gamma}$ be a multiresolution (orthonormal) dictionary that satisfies \eqref{eq-quasi-diag-gamma}. Let Assumptions \ref{ass-coherence-bound-D} and \ref{ass-prob-low-est} be satisfied. Fix $j_0\in\bN$ and set $M\coloneqq|\Lambda_{j_0}|$.

Let $x^{\dagger}\in\ell^2(\Gamma)$ be such that $\|P_{\Lambda_{j_0}}^{\perp}x^{\dagger}\|\leq r$. Consider $m$ i.i.d.\ samples $t_1,\dots,t_m\in\cD$ drawn from the probability distribution $\nu$. Let $y_k \coloneqq F_{t_k}\Phi^* x^{\dagger}+\varepsilon_k$, for $k=1,\dots,m$, with $\max_k \|\varepsilon_k\|_{\cHt}\leq \beta$.

There exist constants $C_0,C_1,C_2,C_4>0$, which depend only on the quasi-diagonalization bounds in \eqref{eq-quasi-diag-gamma} and on $C_F$, such that the following holds. Let $\zeta\in[0,1]$ and $W\coloneqq\diag(2^{bj})_{(j,n)\in \Lambda_{j_0}}$. Let $\widehat{x}$ be a solution of the following minimization problem:
\begin{equation}
\label{eq-dir-min-probl}
    \min_{x\in\ell^2(\Lambda_{j_0})} \|W^{-\zeta}x\|_{1,\omega} :\quad \frac{1}{m}\sum_{k=1}^m \|F_{t_k}\Phi^*\iota_{\Lambda_{j_0}}x-y_k\|_{\cHt}^2  \leq \lc \beta + C_3 2^{-bj_0}r \rc^2,
\end{equation}
where $C_3=c_{\nu}^{-1/2}C_4$.

Fix $s\in[M]$ with $s\geq\|\omega\|_{\infty}^2/4$ and set
\begin{equation}
\label{eq:dir-min-probl-tau}
    \tau\coloneqq B^2 \max_{(j,n)\in\Lambda_{j_0}}(d_{j,n}^{-2} 2^{2bj}) 2^{2(1-\zeta) bj_0}s.
\end{equation}
If
\begin{equation}
\label{eq:thm-interp-samplecomplexity}
    m \geq C_0\tau\max\{\log^3{\tau}\log{M},\log(1/\gamma)\},
\end{equation}
then, with probability exceeding $1-\gamma$, the following recovery estimate holds:
\begin{equation}
    \|W^{-\zeta}x^{\dagger}-W^{-\zeta}\widehat{x}\|_2 \leq C_1 \frac{\sigma_s(P_{\Lambda_{j_0}}W^{-\zeta}x^{\dagger})_{1,\omega}}{\sqrt{s}} + C_2 2^{-\zeta b j_0} c_{\nu}^{-1/2} (2^{b j_0}\beta+C_3 r).
\end{equation}
In particular, we have that
\begin{equation}
    \|x^{\dagger}-\widehat{x}\|_2 \leq C_1 2^{\zeta b j_0} \frac{\sigma_s(P_{\Lambda_{j_0}}W^{-\zeta}x^{\dagger})_{1,\omega}}{\sqrt{s}} + C_2 c_{\nu}^{-1/2} \lc 2^{b j_0}  \beta+ C_3 r \rc.
\end{equation}
\end{theorem}

\begin{remark}
\label{rem:M-est}
    When $(\phi_{j,n})_{j,n}$ is a wavelet-like dictionary, the parameter $M=|\Lambda_{j_0}|$ in Theorem \ref{thm:interpolation} usually satisfies $M\asymp 2^{c j_0}$ for some positive integer $c$ -- see, for instance, equation \eqref{eq:radon-estimate-number-wavelets} below. In these cases, the quantity $\log{M}$ in the sample complexity can be therefore replaced with $cj_0$.
\end{remark}

\begin{remark}
    Thanks to the quasi-diagonalization property \eqref{eq-quasi-diag-gamma}, the quantities related to the matrix $G$ in Theorem \ref{thm-main-result} assume an explicit form in Theorem \ref{thm:interpolation}. Moreover, thanks to refined information about the coherence given by Assumption \ref{ass-coherence-bound-D} and the flexibility given by the parameter $\zeta$, the sample complexity improves in Theorem \ref{thm:interpolation} from Theorem \ref{thm-main-result}; indeed, the sample complexity in the previous scenario was given by $m=m(\tau)$, where it is possible to show that
    \begin{align*}
        \tau\coloneqq B^2 \|G^{-1}\|^4 s \asymp B^2 2^{4bj_0} s,
    \end{align*}
    which, recalling that $d_{j,n}\geq 1$ and $\zeta\in[0,1]$, is in general larger than
    \begin{align*}
        \tau\coloneqq B^2 \max_{(j,n)\in\Lambda_{j_0}}(d_{j,n}^{-2} 2^{2bj}) 2^{2(1-\zeta) bj_0}s.
    \end{align*}
    In practice, this will lead to significant improvements in the estimates for the Radon transform, as will be shown in Section \ref{sec-radon}.
\end{remark}

The tunable parameter $\zeta$ in Theorem~\ref{thm:interpolation} shows that there is a trade-off between the sample complexity of the problem and the reconstruction error: larger values of $\zeta$ correspond to a lower number of measurements, but to a possibly larger reconstruction error.
We now discuss two extremal cases of special interest of Theorem~\ref{thm:interpolation} in the choice of the parameter $\zeta$.
First, we consider the case $\zeta=0$, which yields the largest number of measurements, but possibly the best bound for the reconstruction error. 

\begin{corollary}
\label{cor:main-result-D}
Consider the setting of Theorem \ref{thm:interpolation}, and suppose that the same hypotheses are satisfied. Let $\widehat{x}$ be a solution of the following minimization problem:
\begin{equation}
\label{eq:main-result-D-minprobl}
    \min_{x\in\ell^2(\Lambda_{j_0})} \|x\|_{1,\omega}\quad :\quad  \frac{1}{m}\sum_{k=1}^m \|F_{t_k}\Phi^*\iota_{\Lambda_{j_0}}x-y_k\|_{\cHt}^2 \leq \lc \beta + C_3 2^{-bj_0}r\rc^2.
\end{equation} 
If
\begin{equation}
\label{eq:relcoherence}
    \tau\coloneqq B^2 \max_{(j,n)\in\Lambda_{j_0}}(d_{j,n}^{-2} 2^{2bj}) 2^{2bj_0} s.
\end{equation}
and
\begin{equation}
    m \geq C_0\tau\max\{\log^3{\tau}\log{M},\log(1/\gamma)\},
\end{equation}
then, with probability exceeding $1-\gamma$, the following recovery estimate holds:
\begin{equation}
    \|x^{\dagger}-\widehat{x}\|_2 \leq C_1 \frac{\sigma_s(P_{\Lambda_{j_0}}x^{\dagger})_{1,\omega}}{\sqrt{s}} + C_2 c_{\nu}^{-1/2} (2^{bj_0}\beta + C_3 r ).
\end{equation}
\end{corollary}

Next, we consider the case $\zeta=1$ in Theorem~\ref{thm:interpolation}, which yields the smallest number of measurements, but a possibly larger bound for the reconstruction error.

\begin{corollary}
\label{cor:dir-recovery}
Consider the setting of Theorem \ref{thm:interpolation}, and suppose that the same hypotheses are satisfied. Let $\widehat{x}$ be a solution of the following minimization problem:
\begin{equation}
    \min_{x\in \ell^2(\Lambda_{j_0})} \|W^{-1}x\|_{1,\omega} :\quad  \frac{1}{m}\sum_{k=1}^m \|F_{t_k}\Phi^*\iota_{\Lambda_{j_0}}x-y_k\|_{\cHt}^2  \leq \lc \beta + C_3 2^{-bj_0}r \rc^2,
\end{equation}
If
\begin{equation}
\label{eq:tau_def_relcoherence}
    \tau\coloneqq B^2\max_{(j,n)\in\Lambda_{j_0}}(d_{j,n}^{-2} 2^{2bj})s
\end{equation}
and
\begin{equation}
\label{eq:m_def_relcoherence}
    m \geq C_0\tau\max\{\log^3{\tau}\log{M},\log(1/\gamma)\},
\end{equation}
then, with probability exceeding $1-\gamma$, the following recovery estimate holds:
\begin{equation}
    \|F\Phi^* x^{\dagger}-F\Phi^*\widehat{x}\|_{L_{\mu}^2(\cD;\cHt)} \leq C_1 \frac{\sigma_s( P_{\Lambda_{j_0}} W^{-1}x^{\dagger})_{1,\omega}}{\sqrt{s}} + C_2 c_{\nu}^{-1/2} (\beta+2^{-b j_0}C_3 r).
\end{equation}
In particular, we have that
\begin{equation}
    \|x^{\dagger}-\widehat{x}\|_2 \leq C_1 2^{bj_0} \frac{\sigma_s( P_{\Lambda_{j_0}} W^{-1}x^{\dagger})_{1,\omega}}{\sqrt{s}} + C_2 c_{\nu}^{-1/2} (2^{bj_0} \beta+C_3 r).
\end{equation}
\end{corollary}

As a direct consequence of Corollary~\ref{cor:dir-recovery}, we obtain the following exact recovery result. If there is no truncation error ($x^\dagger\in\Lambda_{j_0}$, i.e.\ $r=0$), if $ x^\dagger$ is $s$-sparse ($\sigma_s\lc P_{\Lambda_{j_0}}W^{-1}{x^\dagger} \rc_1= \sigma_s\lc {x^\dagger} \rc_1=0$) and if there is no noise ($\beta=0$), then
\[
\widehat x=x^\dagger,
\]
provided that $m$ satisfies \eqref{eq:m_def_relcoherence}.

\begin{remark}
\label{rem:rel-coherence}
Let us comment on the sample complexity appearing in \eqref{eq:tau_def_relcoherence} and \eqref{eq:m_def_relcoherence}.
Consider for simplicity the case where $f_{\nu} = 1$, $\omega_{j,n} = 1$ and $\zeta=1$. In classical CS, where $(F_t)_{t\in\cD}$ are the measurement operators and $F$ is an isometry, a relevant quantity appearing in the sample complexity is the \textit{coherence} $B>0$ of the system, such that
\begin{equation}
    \max_{(j,n)\in\Lambda_{j_0}} \|F_t\phi_{j,n}\|_{\cHt} \leq B.
\end{equation}
When a source of ill-posedness is introduced (e.g., if $F$ is a compact operator), this quantity must be suitably normalized with the $L^2$ norm of $F\phi_{j,n}$. Note that in this setting Assumption \ref{ass-coherence-bound-D} reads $\|F_t\phi_{j,n}\|_{\cHt} \leq B d_{j,n}^{-1}$. On the other hand, the quasi-diagonalization property implies that $\|F\phi_{j,n}\|_{\cHd} \asymp 2^{-bj}$. We are led to consider the notion of relative coherence of the system, that is
\begin{equation}
    \max_{(j,n)\in\Lambda_{j_0}} (B d_{j,n}^{-1} 2^{bj}).
\end{equation}
It is evident from \eqref{eq:tau_def_relcoherence} and \eqref{eq:m_def_relcoherence} that the relative coherence of the system plays the same role of the usual coherence in classical CS, where the recovery guarantees are given for $m\gtrsim B^2 s$, up to logarithmic terms.
\end{remark}

\subsection{Analysis of the recovery error}\label{sub:analysis_recovery}
We now derive explicit estimates in terms of the noise $\beta$ and the number of measurements $m$. To accomplish our goal, we need to impose additional assumptions on the growth of the parameters $d_{j,n}$ in Assumption \ref{ass-coherence-bound-D} and on the regularity of the signal $x^{\dagger}$. The latter includes a certain decay of the truncation error $\|P_{\Lambda_{j_0}}^{\perp}x^{\dagger}\|_2$ in terms of $|\Lambda_{j_0}|$ -- namely, $\|P_{\Lambda_{j_0}}^{\perp}x^{\dagger}\|_2\leq C2^{-aj_0}$, where $|\Lambda_{j_0}|\asymp 2^{cj_0}$ -- and of \textit{$p$-compressibility} for some $p\geq 1/2$, which means that $\sigma_s(x^{\dagger})_1\leq Cs^{1/2-p}$.

Those estimates are natural in many problems where the signal $x^{\dagger}$ is assumed to belong to a specific class of signals. For instance, in the case of cartoon-like images, we have $a=1/2$ and approximate $p$-compressibility with $p=1/2$, in a sense to be specified -- see Example \ref{ex:cartoon-like}.

\begin{theorem}
\label{thm:exp-est}
    Consider the setting of Theorem \ref{cor:main-result-D} and suppose that the same hypotheses are satisfied.
    Suppose that the following estimates hold:
    \begin{align*}
        |\Lambda_j|\leq C2^{cj},\quad&\quad
        d_{j,n} \geq 2^{dj},\\
        \sigma_s(x^{\dagger})_1 \leq C s^{1/2-p},\quad&\quad
        \|P_{\Lambda_j}^{\perp}x^{\dagger}\|_2 \leq C 2^{-aj},
    \end{align*}
    for some $a,c,p,C>0$ and $0\leq d\leq b$.

    Fix $\gamma\in(0,1)$ and let $\beta\in(0,1)$ be sufficiently small. Let $j_0\coloneqq\lfloor 1/(a+b)\log(1/\beta) \rfloor$ and let $m$ be sufficiently large. 
    
    There exist constants $C_0,C_1>0$  depending only on $a,b,c,p,C$, on $B$ appearing in Assumption \ref{ass-coherence-bound-D}, on the quasi-diagonalization bounds in \eqref{eq-quasi-diag-gamma}, on $C_F$ and logarithmically on $\gamma$, such that the following holds.
    
    Let $\widehat{x}$ be a solution of the minimization problem \eqref{eq:main-result-D-minprobl}. With probability exceeding $1-\gamma$, the following recovery estimate holds:
    \begin{align*}
        \|x^{\dagger}-\widehat{x}\|_2 \leq C_1 \lc \frac{ \log^{4/p}(1/\beta) }{\beta^{2 \frac{2b-d}{a+b} p} m^p} + \beta^{\frac{a}{a+b}}\rc.
    \end{align*}

    In particular, if $$m=\lfloor C_0\beta^{-2\frac{(2b-d)}{a+b}-\frac{a}{p(a+b)}}\log^4(1/\beta) \rfloor,$$ then
    \begin{align*}
        \|x^{\dagger}-\widehat{x}\|_2 \leq C_1 \beta^{\frac{a}{a+b}},
    \end{align*}
    or, equivalently,
    \begin{align*}
        \|x^{\dagger}-\widehat{x}\|_2 \leq C_1\lc\frac{ \log^4{m} }{m}\rc^{ \frac{ap}{a+4bp-2dp}}.
    \end{align*}
\end{theorem}

\begin{remark}
    Following a favored approach in the statistical inverse learning literature -- see \cite{blanchard}, for instance -- and for simplicity of the exposition, we adopted an asymptotic formulation in Theorem \ref{thm:exp-est}, where $\beta$ is sufficiently small and $m$ is sufficiently large. Analyzing the proof of the Theorem, it is possible to make this statement quantitatively precise.
\end{remark}

\section{The sparse Radon transform}
\label{sec-radon}
We now consider the problem of reconstructing a function supported in the unit ball $\cB_1\subset\bR^2$ by sampling either the corresponding Radon or the fan beam transform along angles $\theta\in\bS^1$.

\subsubsection*{The Radon and the fan beam transform}
Let us recall the definition of the two-dimensional Radon transform.
\begin{definition}
Let $K\subset\bR^2$ be a compact set. Let $u\in L^2(K)$ and $\theta\in[0,2\pi)$. The \textit{Radon transform along the angle $\theta$}, $\cR_\theta\colon L^2(K)\rightarrow L^2(\bR)$, is defined by
\begin{equation}
    \cR_\theta{u}(s) \coloneqq \int_{e_\theta^{\perp}} u(y+s e_\theta) \diff{y},
\end{equation}
where $\diff{y}$ is the 1D Lebesgue measure on $e_\theta^{\perp}$ and $e_{\theta}\coloneqq(\cos{\theta},\sin{\theta})$.

The \textit{(full) Radon transform} $\cR\colon L^2(K)\rightarrow L^2([0,2\pi)\times\bR)$ is defined by $\cR{u}(\theta,s) \coloneqq \cR_\theta{u}(s)$.
\end{definition}
With an abuse of notation, we will write $\cR_{\theta}$ even when $\theta\in\bS^1$, with the natural identification between $[0,2\pi)$ and $\bS^1$.

To be precise, $\cR_{\theta}u$ is defined for functions in a dense subspace of $L^2(K)$ at first, in such a way that the integral expression makes sense, for instance $C^0(K)$. The continuous extension to $L^2(K)$ and the well-posedness of both $\cR_\theta$ and $\cR$ are ensured by the following result. 
\begin{lemma}[{\cite[Theorem 1.6]{natterer}}]
\label{thm-radon-continuity}
$\cR_{\theta}$ and $\cR$ are well-defined continuous operators. Moreover, the norms of $\cR$ and $\{\cR_{\theta}\}_{\theta\in\bS^1}$ can be uniformly bounded by a constant depending only on $|K|$.
\end{lemma}

We also consider the fan beam transform, which is more relevant in the applications. 
\begin{definition}
    Let $0<d<\rho$. Let $u\in L^2(\cB_d)$ and $\theta\in[0,2\pi)$. The \textit{fan beam transform  from the angle $\theta\in[0,2\pi)$}  at distance $\rho$, $\cD_{\theta}\colon L^2(\cB_d)\rightarrow L^2(-\frac\pi 2,\frac\pi 2)$, is defined by
    \begin{align*}
        \cD_{\theta}u(\alpha) \coloneqq \int_{\bR} u(\rho e_{\theta}+te_{\theta+\alpha})\diff{t}.
    \end{align*}
    The \textit{(full) fan beam transform} $\cD\colon L^2(\cB_d)\rightarrow L^2([0,2\pi)\times (-\frac\pi 2,\frac\pi 2) )$ is defined by $\cD u(\theta,\alpha)\coloneqq \cD_{\theta}u(\alpha)$.
\end{definition}
This is nothing but a different parametrization of the Radon transform. We consider functions whose support is contained in the ball of radius $d$, and take radial lines with centers on the sphere of radius $\rho$.

\subsubsection*{Compactly supported wavelets}
\label{sec:radon-wavelets}
Let us briefly review a standard construction of an orthonormal basis $(\phi_{j,n})_{j,n}$ of $L^2(\bR^2)$ consisting of compactly supported wavelets -- see \cite{Da} for a comprehensive discussion.

We start with a preliminary construction in $L^2(\bR)$. For $j\in\bN$ and $n\in\bZ$, given a compactly supported scaling function $\chi\in L^2(\bR)$ and a mother wavelet $\psi\in L^2(\bR)$, we define
\begin{align*}
    \chi_n \coloneqq \tau_n\chi,\qquad
    \psi_{j,n} \coloneqq D_{j-1}^{1}\tau_n\psi,
\end{align*}
where $D_j^{1} f\coloneqq 2^{j/2} f(2^j \cdot)$ is a dyadic dilation of $f$ in 1D and $\tau_n f\coloneqq f(\cdot-n)$ denotes a translation by $n$. 

For an appropriate choice of $\chi$ and $\psi$, we have that $\{\chi_n\}_n\cup\{\psi_{j,n}\}_{j,n}$ is an orthonormal basis of $L^2(\bR)$ -- see \cite[Chapter 6]{Da}. Moreover, as shown in \cite[Theorem 3 in Section 3.8]{Me}, $\chi$ and $\psi$ can be taken in $C^2(\bR)$ -- see Remark \ref{rem-regwav-quasidiag} for further details.

We now build an orthonormal basis using the standard construction of orthonormal separable wavelets -- see for instance \cite[Section 7.7]{Ma}. Recall that the tensor product of functions $f,g$ of one variable is defined by $(f\otimes g)(x,y)\coloneqq f(x)g(y)$. For $j\in\bN$ and $n_1,n_2\in\bZ$, we set
\begin{align*}
    \chi_{n_1,n_2} \coloneqq \chi_{n_1}\otimes\chi_{n_2},\qquad
    &\psi_{0,n_1,n_2}^{(1)} \coloneqq \psi_{0,n_1}\otimes\chi_{n_2},\\
    \psi_{0,n_1,n_2}^{(2)} \coloneqq \chi_{n_1}\otimes\psi_{0,n_2},\qquad
    &\psi_{0,n_1,n_2}^{(3)} \coloneqq \psi_{0,n_1}\otimes\psi_{0,n_2}
\end{align*}
and $\psi_{j,n_1,n_2}^{\varepsilon}\coloneqq D_j^{2}\psi_{0,n_1,n_2}^{\varepsilon}$, where $D_j^2 f\coloneqq 2^j f(2^j\cdot)$ is the dyadic dilation of $f$ in 2D. It can be proved that
\begin{align*}
    \{\chi_{n_1,n_2}\}_{n_1,n_2}\cup\{\psi_{j,n_1,n_2}^{(\varepsilon)}\}_{j,n_1,n_2,\varepsilon},\qquad \varepsilon \in \{1,2,3\},
\end{align*}
is an orthonormal basis of $L^2(\bR^2)$.

Finally, we obtain the orthonormal basis $(\phi_{j,n})_{j,n}$ by rearranging the wavelets defined above, according to the following rule: for $j=0$, $n=(n_1,n_2)\in\bZ^2$ and, for $j\geq 1$, $n=(n_1,n_2,\varepsilon)\in\bZ^2\times\{1,2,3\}$, set
\begin{align*}
    \phi_{0,n} \coloneqq \chi_{n_1,n_2};\qquad \phi_{j,n}\coloneqq \psi_{j,n_1,n_2}^{(\varepsilon)}.
\end{align*}
Let $\Gamma$ be the set of indices $(j,n)$ such that $\supp\phi_{j,n}\cap\cB_1\neq \emptyset$, namely,
\[
\Gamma = \{(j,n)\in(\{0\}\times\bZ^2)\cup\left(\bN\times(\bZ^2\times\{1,2,3\})\right):\supp\phi_{j,n}\cap\cB_1\neq \emptyset\}.
\] In what follows, we will consider the space $\cH_1\coloneqq\overline{\Span(\phi_{j,n})_{(j,n)\in\Gamma}}$. Notice that $L^2(\cB_1)\subseteq\cH_1$. It is easy to realize that
\begin{align}
\label{eq:radon-estimate-number-wavelets}
    |\{(j',n)\in\Gamma\colon\ j'=j\}|\asymp 2^{2j},\qquad j\in\bN.
\end{align}
Indeed, by construction, the number of wavelets at scale $j+1$ is approximately given by the corresponding number at scale $j$ multiplied by $2^2$. In light of \eqref{eq:Lambda_j}, this implies that $|\Lambda_j|\asymp 2^{2j}$.

Finally, we choose $0<d<\rho$ such that $\supp{u}\subseteq\cB_d$ for every $u\in\cH_1$, so that $\cD u$ is well-defined for every  $u\in\cH_1$.

\subsubsection*{Main results}
Our main recovery results for the sparse Radon transform read as follows.
\begin{theorem}
\label{thm:radon-main-thm}
    Consider a wavelet dictionary $(\phi_{j,n})_{j,n\in\Gamma}$ defined as above and the corresponding analysis operator $\Phi\colon \cH_1\rightarrow\ell^2(\Gamma)$. Let $F$ be either $\cR$ or $\cD$. There exist constants $C_0,C_1,C_2,C_3>0$, which depend only on the chosen wavelet basis in the case of the Radon transform and also on $d$ and $\rho$ in the case of the fan beam transform, such that the following holds.

     Fix $j_0\in\bN$. Let $u^{\dagger}\in L^2(\cB_1)$ be such that $\|P_{\Lambda_{j_0}}^{\perp} \Phi u^{\dagger}\|_2\leq r$. Consider $m$ i.i.d.\ samples $\theta_1,\dots,\theta_m\in[0,2\pi)$ drawn from the uniform distribution on $[0,2\pi)$.
    
    Let $y_k \coloneqq F_{\theta_k} u^{\dagger}+\varepsilon_k$ for $k=1,\dots,m$, with $\max_k \|\varepsilon_k\|_{L^2} \leq \beta$. Let $\widehat{u}$ be a solution of the minimization problem
    \begin{align}
    \label{eq:radon-min-prob}
        \min_{u} \|\Phi{u}\|_1\quad\colon\quad \frac{1}{m} \sum_{k=1}^m \|F_{\theta_k}u-y_k\|_{L^2}^2 \leq
        \lc \beta+ C_3 2^{-j_0/2} r \rc^2,
    \end{align}
    where the minimum is taken over $\Span(\phi_{j,n})_{(j,n)\in\Lambda_{j_0}}$.
    
    Let $s\in[|\Lambda_{j_0}|]$, $s\geq 2$, and set $\tau\coloneqq 2^{j_0}s$. For any $\gamma\in(0,1)$, if
    \begin{align*}
        m\geq C_0\tau\max\{j_0\log^3{\tau},\log(1/\gamma)\},
    \end{align*}
    then, with probability exceeding $1-\gamma$, the following recovery estimate holds:
    \begin{equation}\label{eq:Radon_error_1}
        \|u^{\dagger}-\widehat{u}\|_{L^2} \leq
        C_1\frac{\sigma_s(P_{\Lambda_{j_0}}\Phi{u^{\dagger}})_1}{\sqrt{s}} + C_2 ( 2^{j_0/2}\beta + r ).
    \end{equation}
\end{theorem}

The above result corresponds to the case $\zeta=0$ in Theorem~\ref{thm:interpolation} (see Corollary~\ref{cor:main-result-D}), in which a smaller reconstruction error is obtained by making more measurements. In the next theorem, we consider the case with a minimum number of measurements (i.e., the case where $\zeta=1$, considered in Corollary~\ref{cor:dir-recovery}), proportional to the sparsity $s$ up to logarithmic factors. This is best suited for the recovery of $F u^\dagger$ and for the exact recovery of $u^\dagger$ in absence of noise, but yields worse error bounds in general.

\begin{theorem}
\label{thm:radon-dir-prob}
    Consider the setting of Theorem \ref{thm:radon-main-thm} and suppose that the same assumptions are satisfied. 

    There exist constants $C_0,C_1,C_2>0$, which depend only on the chosen wavelet basis in the case of the Radon transform and also on $d$ and $\rho$ in the case of the fan beam transform, such that the following holds. Let $W\coloneqq\diag(2^{j/2})_{j,n}$. Let $\widehat{u}$ be a solution of the minimization problem
    \begin{align*}
        \min_u \|W^{-1}\Phi{u}\|_1\quad\colon\quad 
         \frac{1}{m} \sum_{k=1}^m \|F_{\theta_k}u-y_k\|_{L^2}^2 
        \leq
        \lc \beta+ C_3 2^{-j_0/2} r \rc^2,
    \end{align*}
    where the minimum is taken over $\Span(\phi_{j,n})_{(j,n)\in\Lambda_{j_0}}$.
    
    Let $s\in[|\Lambda_{j_0}|]$, $s\ge 2$. If
    \begin{align} \label{eq:m_radon_optimal}
        m \geq C_0 s\max\{j_0\log^3{s},\log(1/\gamma)\},
    \end{align}
    then, with probability exceeding $1-\gamma$, the following recovery estimate holds:
    \begin{align*}
        \|F{u}^{\dagger}-F{\widehat{u}}\|_{L^2(\bS^1\times\bR)} \leq C_1 \frac{\sigma_s\lc P_{\Lambda_{j_0}}W^{-1}\Phi{u^\dagger} \rc_1}{\sqrt{s}} + C_2 (\beta+2^{-j_0/2}r).
    \end{align*}
    In particular, we have that
    \begin{equation}\label{eq:Radon_error_2}
        \|u^{\dagger}-\widehat{u}\|_{L^2} \leq C_1 2^{j_0/2} \frac{\sigma_s\lc P_{\Lambda_{j_0}}W^{-1}\Phi{u^\dagger} \rc_1}{\sqrt{s}} + C_2 (2^{j_0/2}\beta+r).
    \end{equation}
\end{theorem}

As anticipated, in Theorem~\ref{thm:radon-dir-prob} we have a lower number of measurements than in Theorem~\ref{thm:radon-main-thm}. On the other hand, the error estimate \eqref{eq:Radon_error_2} is possibly worse than the estimate \eqref{eq:Radon_error_1}, because of the quantity $2^{j_0/2}$ multiplying the first term of the right hand side. However, it should be observed that $\sigma_s\lc P_{\Lambda_{j_0}}W^{-1}\Phi{u^\dagger} \rc_1\le \sigma_s\lc P_{\Lambda_{j_0}}\Phi{u^\dagger} \rc_1$, and so the comparison is not straightforward.

As a direct consequence of Theorem~\ref{thm:radon-dir-prob}, we obtain the following exact recovery result. If there is no truncation error ($u^\dagger\in\Span(\phi_{j,n})_{(j,n)\in\Lambda_{j_0}}$, i.e., $r=0$), if $\Phi u^\dagger$ is $s$-sparse ($\sigma_s\lc P_{\Lambda_{j_0}}W^{-1}\Phi{u^\dagger} \rc_1= \sigma_s\lc \Phi{u^\dagger} \rc_1=0$) and if there is no noise ($\beta=0$), then
\[
\widehat u=u^\dagger,
\]
provided that $m$ satisfies \eqref{eq:m_radon_optimal}.

As in Section~\ref{sub:analysis_recovery}, 
we now derive explicit estimates in terms of the noise $\beta$ and the number of measurements $m$.

\begin{theorem}
\label{thm:radon-exp-est}
    Consider the setting of Theorem \ref{thm:radon-main-thm} and suppose that the same hypotheses are satisfied.
    
    Suppose that the following estimates hold:
    \begin{align*}
        \sigma_s(\Phi u^{\dagger})_1 \leq C s^{1/2-p},\qquad
        \|P_{\Lambda_j}^{\perp} \Phi u^{\dagger}\|_2 \leq C 2^{-aj},
    \end{align*}
    for some $a,p,C>0$.

    Fix $\gamma\in(0,1)$ and let $\beta\in(0,1)$ be sufficiently small. Let $j_0\coloneqq\lfloor 2/(2a+1)\log(1/\beta) \rfloor$ and let $m$ be sufficiently large.

    There exist constants $C_0,C_1>0$  depending only on the chosen wavelet basis and on $a,p,C$, logarithmically on $\gamma$ (and on $d$ and $\rho$ in the case of the fan beam transform) such that the following holds.
    
 Let $\widehat{x}$ be a solution of the minimization problem \eqref{eq:radon-min-prob}. Then, with probability exceeding $1-\gamma$, the following recovery estimate holds:
    \begin{align*}
        \|u^{\dagger}-\widehat{u}\|_{L^2} \leq C_1 \lc \frac{ \log^{4/p}(1/\beta) }{\beta^{ \frac{2p}{2a+1}} m^p} + \beta^{\frac{2a}{2a+1}}\rc.
    \end{align*}

    In particular, if $$m=\lfloor C_0\beta^{-\frac{2}{2a+1}-\frac{2a}{p(2a+1)}}\log^4(1/\beta) \rfloor.$$ then
    \begin{align*}
        \|u^{\dagger}-\widehat{u}\|_{L^2} \leq C_1 \beta^{\frac{2a}{2a+1}},
    \end{align*}
    or, equivalently,
    \begin{align*}
        \|u^{\dagger}-\widehat{u}\|_{L^2} \leq C_1\lc\frac{ \log^4{m} }{m}\rc^{ \frac{ap}{a+p}}.
    \end{align*}
\end{theorem}

Analogous estimates have been obtained for the sparse Radon transform in the context of statistical inverse learning \cite{bubba-burger-helin-ratti-2021,bubba-ratti-2022}. These are similar in spirit to the ones presented above, but comparison is non-trivial since the setting is substantially different as sparsity is not considered there, while here it plays a crucial role.

\begin{example}[Cartoon-like images] 
\label{ex:cartoon-like}
    In the case of cartoon-like images supported in $\cB_1$ (see \cite[Section 9.2.4]{Ma} -- in short, $C^2$ signals apart from $C^2$ edges), with reference to the notation of Theorem \ref{thm:radon-exp-est}, we have that $a=1/2$ and $p$-compressibility holds with $p=1/2$ up to logarithmic terms -- see the end of Section~\ref{sec:proof-radon}. There exist constants $C_0,C_1>0$, which only depend on the chosen wavelet basis in the case of the Radon transform and also on $d$ and $\rho$ in the case of the fan beam transform, such that the following bound on the $L^2$ error holds, up to logarithmic terms:
    \begin{align}
    \label{eq:ex-cartoon-like-est}
        \|u^{\dagger}-\widehat{u}\|_{L^2} \leq C_1\lc\frac{1}{\beta^{1/2}m^{1/2}}+\beta^{1/2}\rc.
    \end{align}
    Choosing $m=\lfloor C_0\beta^{-2}\log^4(1/\beta) \rfloor$, we get that, up to logarithmic terms,
    \begin{align*}
        \|u^{\dagger}-\widehat{u}\|_{L^2} \leq C_1\beta^{1/2},\qquad
        \|u^{\dagger}-\widehat{u}\|_{L^2} \leq C_1\lc \frac{1}{m} \rc^{1/4}.
    \end{align*}
    We refer the reader to Section~\ref{sec:proof-radon} for a proof of \eqref{eq:ex-cartoon-like-est}.
\end{example}

\section{Proofs of the main results}\label{sec-proofs}

In this section, we will usually write $X \lesssim Y$ if the underlying inequality holds up to a positive constant $C>0$, that is $X \le CY$. Similarly, we write $X \gtrsim Y$ if $X\ge CY$ for some $C>0$. In particular, $X \asymp Y$ means that both $X \lesssim Y$ and $X \gtrsim Y$ hold. These constants may depend in general on other quantities related to the results from which the inequalities are inferred, the dependence being made explicit in the corresponding statements.

\subsection{Preliminary results}\label{sec-weight-extra}
Determining the best $s$-$\omega$-sparse approximation of a vector with respect to $\ell_{\omega}^p$, introduced in Section \ref{sec-weight-set}, is a quite challenging problem in practice. A surrogate notion of quasi-best sparse approximation has been introduced in \cite[Section 3]{RW}. In fact, the latter construction happens to be powerful enough for the purposes of the proofs given below, hence we briefly review the basic facts in this connection for the benefit of the reader. 

Given $x\in\ell^2(\Gamma)$, $\omega\in[1,+\infty)^{\Gamma}$ and a positive integer $s$, consider the permutation $\pi\colon \Gamma\rightarrow \Gamma$ associated with the non-increasing rearrangement of $(|x_i| \omega_i^{-1})_{i\in \Gamma}$. Let $i_0\in \Gamma$ be the largest index such that
\begin{equation}
    \omega(S_{i_0}) \leq s, \quad S_{i_0} \coloneqq  \{\pi(1),\dots,\pi(i_0)\}.
\end{equation}
We thus define $x_{S_{i_0}}$ to be the \textit{quasi-best s-$\omega$-sparse approximation to $x$}. Notice that this notion does not depend on a $\ell_{\omega}^p$ norm, whereas the companion \textit{error of quasi-best s-$\omega$-sparse approximation to $x\in\ell^2(\Gamma)$} with respect to $\ell_{\omega}^p$  is defined by $\tilde{\sigma}_s(x)_{p,\omega} \coloneqq  \|x_{S_{i_0}^c}\|_{p,\omega}$. From the very definition we have that
\begin{equation}
    \sigma_s(x)_{p,\omega} \leq \tilde{\sigma}_s(x)_{p,\omega}.
\end{equation}

For later use we recall the following \textit{Stechkin-type inequality} -- see \cite[Lemma 3.12]{adcock_sparsebook} for a proof. If $s>0$ and $0<p<q\le 2$, then
\begin{equation}
\label{eq-stecchini}
\sigma_s(x)_{q,\omega} \le \tilde{\sigma}_s(x)_{q,\omega} \leq s^{1/q-1/p}\|x\|_{p,\omega},\qquad x\in \ell^2(\Gamma).
\end{equation}

To simplify the notation in the proofs, in what follows we will also introduce the following sampling operator $A\colon\ell^2(\Gamma)\rightarrow\cHt^m$ associated with a sample $(t_1,\dots,t_m)\in\cD^m$:
\begin{align*}
    Ax \coloneqq \frac{1}{\sqrt{m}}\lc F_{t_1}\Phi^* x,\dots,F_{t_m}\Phi^* x \rc.
\end{align*}
With an abuse of notation, we will use $A$ to denote also the restriction $A\colon\ell^2(\Lambda)\rightarrow\cHt^m$ to some finite $\Lambda\subset\Gamma$.

Moreover, with reference to problem \eqref{eq-min-probl1}, we define $\varepsilon\coloneqq \frac{1}{\sqrt{m}}\lc \varepsilon_1,\dots,\varepsilon_m \rc$ and $y\coloneqq \frac{1}{\sqrt{m}}\lc y_1,\dots,y_m\rc$.
Notice that the condition $\max_k \|\varepsilon_k\|_{\cHt}\leq \beta$ implies that
\begin{align*}
    \|\varepsilon\|_{\cHt^m}^2 =
    \frac{1}{m} \sum_{k=1}^m \|\varepsilon_k\|_{\cHt}^2 \leq \beta^2.
\end{align*}
Moreover, the following identity holds for $x\in\ell^2(\Lambda)$:
\begin{align*}
    \|Ax-y\|_{\cHt^m}^2 = \frac{1}{m}\sum_{k=1}^m \|F_{t_k}\Phi^*\iota_{\Lambda}x-y_k\|^2,
\end{align*}
where $y_k = F_{t_k}\Phi^*x^{\dagger}+\varepsilon_k$, for $k=1,\dots,m$, as defined in Section~\ref{sec-setting}.

\subsection{The robust null space property}\label{sec-rnsp} 

The typical strategy to obtain recovery estimates in a RIP-based setting requires to initially introduce a suitably designed version of the \textit{robust null space property} (RNSP), the latter being crucial to ensure stable distance bounds and thus uniform recovery results. We recall that all the Hilbert spaces are assumed to be complex and separable, unless otherwise noted.

\begin{definition}\label{def-rnsp}
Let $\cH$ be a Hilbert space and let $\Lambda$ be a finite or countable set. Consider a vector of weights $\omega\in[1,+\infty)^{\Lambda}$, an integer $s\ge 1$, and real parameters $0<\rho<1$, $\kappa>0$. An operator $A\in\cL(\ell^2(\Lambda),\cH^m)$ is said to satisfy the robust null space property (RNSP) with respect to $(\omega,\rho,\kappa,s)$ if
\begin{align*}
    \|x_S\|_2 \leq \frac{\rho}{\sqrt{s}}\|x_{S^c}\|_{1,\omega}+\kappa\|Ax\|_{\cH^m},\quad \forall \, x\in\ell^2(\Lambda),\ S\subset\Lambda : \omega(S)\leq s.
\end{align*}
\end{definition}
\begin{remark}
\label{rem-rnsp}
A sufficient condition in order for $A$ to satisfy the RNSP with respect to $(\omega,\rho,\kappa,s)$ can be stated as follows: there exists $0<\rho'<1/2$ and $\kappa'>0$ such that
\begin{equation}
\label{eq-rem-rnsp}
    \|x_S\|_2\leq\frac{\rho'}{\sqrt{s}}\|x\|_{1,\omega}+\kappa'\|Ax\|_{\cH^m},\quad \forall \, x\in\ell^2(\Lambda),\ S\subset\Lambda : \omega(S)\leq s.
\end{equation} This is indeed an easy consequence of the estimate
	\[ \| x \|_{1,\omega} = \| x_S \|_{1,\omega} + \| x_{S^c}\|_{1,\omega} \le \sqrt{s} \| x_S \|_2 + \| x_{S^c}\|_{1,\omega}, \]
	which follows in turn by the Cauchy-Schwarz inequality. The relationships between the involved constants are given by
	\[ \rho = \frac{\rho'}{1-\rho'}, \quad \kappa = \frac{\kappa'}{1-\rho'}. \]
\end{remark}

We now show that the RNSP implies suitable distance bounds. Indeed, the proof of the following results is largely inspired by standard arguments in the theory of compressed sensing -- see for instance \cite[Theorem 5]{F} and \cite[Theorem 4.2]{RW}. The proof is essentially identical to \cite[Lemma 6.24]{adcock2022efficient}; adapting it to the case of Hilbert-valued measurement requires only some minor modifications.

\begin{theorem}[\textbf{RNSP $\Rightarrow$ distance bounds}]
\label{thm-rnsp-rec}
Let $\cH$ be a Hilbert space and assume that $A\in\cL(\ell^2(\Lambda),\allowbreak\cH^m)$ satisfies the RNSP with respect to $(\omega,\rho,\kappa,s)$ as in Definition~\ref{def-rnsp}. Then, for all $x,z\in\ell_{\omega}^1(\Lambda)$, 
\begin{align*}
    \|x-z\|_{1,\omega}\leq c_1'\lc \|z\|_{1,\omega}-\|x\|_{1,\omega}+2\sigma_s(x)_{1,\omega} \rc + c_2'\sqrt{s}\|A(x-z)\|_{\cH^m},
\end{align*}
where we set
\begin{align*}
    c_1' = \frac{1+\rho}{1-\rho},\quad c_2'= \frac{2\kappa}{1-\rho}.
\end{align*}
Moreover
\begin{align*}
    \|x-z\|_2\leq \frac{c_1}{\sqrt{s}}\lc \|z\|_{1,\omega}-\|x\|_{1,\omega}+2\sigma_s(x)_{1,\omega} \rc + c_2\|A(x-z)\|_{\cH^m},
\end{align*}
where
\begin{align*}
    c_1 = (\rho+1)c_1',\quad c_2 =(\rho+1)c_2'+\kappa.
\end{align*}
\end{theorem}

\subsection{The generalized restricted isometry property}
\label{sec-grip}
We now define a generalized version of the standard restricted isometry property (RIP), which corresponds to the G-RIPL without levels, and with a weighted sparsity, introduced in \cite[Definition 3.5]{AAH}.

\begin{definition}[\textbf{g-RIP}]\label{def-grip}
Let $\cH$ be a Hilbert space. Given $G\in\cL(\ell^2(\Lambda))$, $0\leq\delta<1$, $\lambda>0$ and a vector of weights $\omega\in[1,+\infty)^{\Lambda}$, we say that $A\in\cL(\ell^2(\Lambda),\cH^m)$, $m\in\bN$, satisfies the \textit{generalized restricted isometry property (g-RIP)} with respect to $(G,\omega,\delta,\lambda)$ if
\begin{equation}
\label{eq-grip}
    (1-\delta)\|Gx\|_2^2 \leq \|Ax\|_{\cH^m}^2 \leq (1+\delta)\|Gx\|_2^2,\quad x\in\ell^2(\Lambda),\ \|x\|_{0,\omega}\leq\lambda.
\end{equation}
\end{definition}

\begin{remark}[\textbf{Diagonal invariance of the g-RIP}]
\label{rem-diag-inv}
    Consider the case where $\Lambda$ is finite. Let $Z\coloneqq(z_i)_{i\in\Lambda}$ be a diagonal matrix with $z_i\neq 0$. Since $\supp(x')=\supp(Zx')$, we have $\|x'\|_{0,\omega}=\|Zx'\|_{0,\omega}$. As a result, after the change of variable $x = Zx'$, it is clear that $A$ satisfies the g-RIP with respect to $(G,\omega,\delta,\lambda)$ if and only if $AZ$ satisfies the g-RIP with respect to $(GZ,\omega,\delta,\lambda)$.
\end{remark}
We now provide sufficient conditions under which the g-RIP implies the RNSP. We also refer the reader to \cite[Theorem 5.5]{AAH} for a version of this result adapted to sparsity in levels.

\begin{theorem}[\textbf{g-RIP $\Rightarrow$ RNSP}]
\label{thm-grip-rnsp}
Fix $\rho'\in(0,1/2)$ and $\delta\in(0,1)$. Assume that $A\in\cL(\ell^2(\Lambda),\cH^m)$ satisfies the g-RIP with respect to $(G,\omega,\delta,\lambda)$ for some invertible $G\in \cL(L^2(\Lambda))$, weights $\omega\in[1,+\infty)^{\Lambda}$ and $\lambda\in\bR^+$. For any $s\geq\|\omega\|_{\infty}^2/4$, if
\begin{equation}
\label{eq-thm-grip-rnsp-02}
    \lambda \geq  \frac{5}{\rho'^2}\frac{1+\delta}{1-\delta}\|G^{-1}\|^2 \|G\|^2 s,
\end{equation}
then $A$ satisfies the RNSP with respect to $(\omega,\rho,\kappa,s)$, where
\begin{align*}
    \rho =\frac{\rho'}{1-\rho'},\quad \kappa = \frac{\|G^{-1}\|}{(1-\rho')\sqrt{1-\delta}}.
\end{align*}
\end{theorem}
\begin{proof}
In view of Remark \ref{rem-rnsp}, it suffices to prove that \eqref{eq-rem-rnsp} is satisfied. 

Fix $x\in\ell^2(\Lambda)$ and $S\subset\Lambda$ with $\omega(S)\leq s$. Consider $x_{S^c}$. Having in mind the construction of the quasi-best $\lambda$-$\omega$-sparse approximation to $x$ in Section \ref{sec-weight-extra}, we now proceed as follows. 

Let $\pi\colon \bN\rightarrow S^c$ be the enumeration yielding the non-increasing rearrangement of $(|x_k|\omega_k^{-1})_{k\in S^c}$, so that $|x_{\pi(i)}|\omega_{\pi(i)}^{-1}\geq |x_{\pi(i+1)}|\omega_{\pi(i+1)}^{-1}$. Set $i_0=0$ and for $j\geq 1$ iteratively define $i_j\in \bN$ to be the largest index such that $i_j>i_{j-1}$ and
\begin{equation}
    \omega(\Omega_j)\leq \tilde \lambda,\quad\Omega_j\coloneqq\{\pi(i_{j-1}+1),\dots,\pi(i_j)\},
\end{equation}
where
\begin{equation}
\label{eq:thm-grip-rnps-03}
    \tilde \lambda = \frac{4}{\rho'^2}\frac{1+\delta}{1-\delta}\|G^{-1}\|^2 \|G\|^2 s.
\end{equation}

Note that $\lambda \ge \tilde \lambda+ \frac{1}{\rho'^2}\frac{1+\delta}{1-\delta}\|G^{-1}\|^2\|G\|^2 s \geq \tilde \lambda+s$ since $\rho' \in (0,1/2)$.
We stress that each $\Omega_j$ is non-empty, thanks to the the fact that $\tilde \lambda \geq 4s \geq\|\omega\|_{\infty}^2$ by assumption, and it is finite since $\omega_i\geq 1$.

It is straightforward to realize that $\cup_{j\geq 1}\Omega_j=S^c$. Moreover, for $j\geq 1$, it follows from the definition that $x_{\Omega_j}$ is the quasi-best $\tilde \lambda$-$\omega$-sparse approximation to $x_{(\Omega_j\cup\Omega_{j+1})}$. The Stechkin-type inequality \eqref{eq-stecchini} gives
\begin{equation}
\label{eq-thm-grip-rnsp-01}
    \|x_{\Omega_j}\|_2 = \tilde{\sigma}_{\tilde \lambda}(x_{(\Omega_{j-1}\cup\Omega_{j})})_{2,\omega} \leq
    \frac{1}{\sqrt{\tilde \lambda}}\|x_{(\Omega_{j-1}\cup\Omega_{j})}\|_{1,\omega},\quad j\geq 2.
\end{equation}
Since $x_S+x_{\Omega_1}=x-\sum_{j\geq 2} x_{\Omega_j}$ by construction, the g-RIP for $A$ yields
\begin{equation}\label{eq:g1}
\begin{split}
    \|x_S\|_2\leq \|x_S+x_{\Omega_1}\|_2 &\leq \|G^{-1}\| \|G(x_S+x_{\Omega_1})\|_2 \\ &\leq
    \frac{\|G^{-1}\|}{\sqrt{1-\delta}} \|A(x_S+x_{\Omega_1})\|_{\cH^m} \\ &\leq
    \frac{\|G^{-1}\|}{\sqrt{1-\delta}} \left(\sum_{j\geq 2} \|Ax_{\Omega_j}\|_{\cH^m} + \|Ax\|_{\cH^m}\right).
\end{split}
\end{equation}
Now, by \eqref{eq-thm-grip-rnsp-01}, we have 
\begin{equation}\label{eq:g2}
    \|Ax_{\Omega_j}\|_{\cH^m} \leq \sqrt{1+\delta} \|Gx_{\Omega_j}\|_2 \leq
    \sqrt{1+\delta} \|G\| \|x_{\Omega_j}\|_2  \leq
    \sqrt{1+\delta} \|G\| \frac{1}{\sqrt{\tilde \lambda}}  \|x_{(\Omega_{j-1}\cup\Omega_{j})}\|_{1,\omega}.
\end{equation}
We have
\begin{align}\label{eq:g3}
    \sum_{j\geq 2} \|x_{(\Omega_{j-1}\cup\Omega_{j})}\|_{1,\omega} \leq
    \sum_{j\geq 2} \|x_{\Omega_{j-1}}\|_{1,\omega} +\|x_{\Omega_{j}}\|_{1,\omega} \leq 2\|x\|_{1,\omega}.
\end{align}
Combining \eqref{eq:g1}, \eqref{eq:g2} and \eqref{eq:g3} we obtain
\begin{equation*}
   \|x_S\|_2 \leq
    \frac{\sqrt{1+\delta}}{\sqrt{1-\delta}} \|G^{-1}\| \|G\| \frac{2\sqrt{s}}{\sqrt{\tilde \lambda}}
    \frac{1}{\sqrt{s}} \|x\|_{1,\omega} + \frac{\|G^{-1}\|}{\sqrt{1-\delta}} \|Ax\|_{\cH^m}.
\end{equation*}
It is now enough to realize that \eqref{eq:thm-grip-rnps-03} implies
\begin{equation}
    \frac{\sqrt{1+\delta}}{\sqrt{1-\delta}} \|G^{-1}\| \|G\| \frac{2\sqrt{s}}{\sqrt{\tilde \lambda}} = \rho'.
\end{equation}
The proof of \eqref{eq-rem-rnsp}, hence of the claim, is concluded. 
\end{proof}

\subsection{Sufficient conditions for the g-RIP via subsampling}\label{sec-sampl}
We now provide sufficient conditions for the g-RIP to be satisfied by the sampling operator of the model.

\begin{theorem}[\textbf{Sampling $\Rightarrow$ g-RIP}]
\label{thm-samp-grip}
Consider the setting of Section \ref{sec-setting}. Suppose that Assumptions \ref{ass-coherence-bound} and \ref{ass-prob-low-est} are satisfied.
Define
\begin{align*}
    G\coloneqq\sqrt{P_{\Lambda}\Phi F^* F\Phi^*\iota_{\Lambda}}.
\end{align*}
Suppose that $G$ is invertible.

Let $(t_1,\dots,t_m)\in\cD^m$,  $m\in\bN$, be independent and identically distributed samples from $\nu$ (with possible repetitions to be kept). Let $A$ be the sampling operator associated with $(t_1,\dots,t_m)$, namely
\begin{equation*}
    A \coloneqq \lc \frac{1}{\sqrt{m}}F_{t_k}\Phi^*\iota_{\Lambda} \rc_{k=1}^m,
\end{equation*}
and introduce the operator
\begin{equation}
\label{eq:q-definition}
    Q\in\cL(\cH^m), \quad Q(h_1,\dots,h_m) = (f_{\nu}(t_1)^{-1/2}h_1,\dots,f_{\nu}(t_m)^{-1/2}h_m).
\end{equation}

Set $M=|\Lambda|$ and $\lambda\in[M]$. Assume that $\tau\coloneqq B^2\|G^{-1}\|^2\lambda\geq 3$. There exists a universal constant $C>0$ such that, for any $\gamma,\delta\in(0,1)$,  if
\begin{equation}
\label{eq-thm-samp-grip-01}
    m \geq C\delta^{-2}\tau\max\{\log^3{\tau}\log{M},\log(1/\gamma)\},
\end{equation}
then $QA$ satisfies the g-RIP with respect to $(G,\omega,\delta,\lambda)$ with probability exceeding $1-\gamma$.
\end{theorem}

The proof of this result follows a well-established pathway in the theory of compressed sensing, aimed at obtaining probabilistic bounds for $\delta$ -- see for instance the arguments and techniques used in the proofs of \cite[Theorem 3.6]{AAH}, \cite[Theorem 3.1]{KNW}, \cite[Theorem 3.2]{LA} and \cite[Theorem 5.2]{RW}. We refer the reader to Appendix \ref{appendix:proofs} for a proof of Theorem \ref{thm-samp-grip}.

\subsection{Proof of Theorem \ref{thm-main-result}}
The last ingredient needed to obtain the proof of Theorem~\ref{thm-main-result} is the following estimate on the error arising from the truncation of the infinite dimensional unknown $x^\dagger$ to the finite-dimensional set $\Lambda$.

\begin{proposition}[\textbf{Truncation error}]
\label{prop:truncation-error}
    Consider the setting of Theorem \ref{thm-main-result}. Suppose that $\|P_{\Lambda}^{\perp}x^{\dagger}\|_2 \leq r$ and
    \begin{align}
    \label{eq:prop-trunc-sample-compl}
        m \geq \|G_1^{-1}\|^2 \|G_2^{-1}\|^2 \|G_2\|^2 \log(1/\gamma)
    \end{align}
    for some invertible $M\times M$ matrices $G_1,G_2$ such that one of these conditions is satisfied:
    \begin{enumerate}
        \item $G_1=G_2$,
        \item $\|G_1^{-1}\|\geq c$ for some constant $c$.
    \end{enumerate}
    Let $Q$ be defined as in \eqref{eq:q-definition}. Then, with probability exceeding $1-\gamma$, we have that
    \begin{align*}
        \|G_2^{-1}\|\|QAP_{\Lambda}^{\perp}x^{\dagger}\|_{\cHt^m} \leq C_1 c_{\nu}^{-1/2} \lc \|F\Phi^*\iota_{\Gamma\setminus\Lambda}\|\|G_2^{-1}\|+1 \rc r,
    \end{align*}
    where $C_1>0$ is a constant depending only on $C_F$  (and on $c$ in the second case).
\end{proposition}
\begin{proof}
Let $u_R\coloneqq \Phi^* P_{\Lambda}^{\perp}x^{\dagger}$ and $c_F\coloneqq \|F\Phi^*\iota_{\Gamma\setminus\Lambda}\|$. We apply Lemma \ref{lem-thm-samp-grip-bernstein} to the random variables $Y_k(t)\coloneqq \|QF_{t_k}u_R\|_{\cHt}^2-\|Fu_R\|_{L_{\mu}^2(\cD;\cHt)}^2$ in the case where the family of functions $f_x$ consists just of the identity mapping; as a consequence, we obtain as a special case the classical version of Bernstein's inequality. Notice that, in this case,
\begin{align*}
    \|QA P_\Lambda^\perp x^\dagger \|_{\cHt^m}^2 = \frac{1}{m} \sum_{k=1}^m f_{\nu}(t_k)^{-1}\|F_{t_k}u_R\|_{\cHt}^2.
\end{align*}

We claim that $\bE(f_{\nu}(t_k)^{-1}\|F_{t_k}u_R\|_{\cHt}^2-\|Fu_R\|_{L_{\mu}^2(\cD;\cHt)}^2)=0$. Indeed, by definition of $F$, we have that 
\begin{align*}
    \int_{\cD} \|F_t u\|_{\cHt}^2 \diff{\mu}(t) = \|Fu\|_{L_\mu^2(\cD;\cHt)}^2,\quad u\in\cH_1,
\end{align*}
which is equivalent to the claim. Moreover, we have that
\begin{align*}
    |f_{\nu}(t_k)^{-1}\|F_{t_k}u_R\|_{\cHt}^2-\|Fu_R\|_{L_{\mu}^2(\cD;\cHt)}^2| &\leq c_{\nu}^{-1}\|F_{t_k}u_R\|_{\cHt}^2+\|Fu_R\|_{L_{\mu}^2(\cD;\cHt)}^2 \\ &\leq 2 C_F^2 c_{\nu}^{-1}r^2 \eqqcolon K,
\end{align*}
where we have used the fact that $\|F_{t}\|,\|F\|\leq C_F$. Finally, we have that
\begin{align*}
    &\bE| f_{\nu}(t_k)^{-1}\|F_{t_k}u_R\|_{\cHt}^2-\|Fu_R\|_{L_{\mu}^2(\cD;\cHt)}^2 |^2\\ &\qquad =
    \bE| f_{\nu}(t_k)^{-1/2} \|F_{t_k}u_R\|_{\cHt} |^4 - \lc\bE\|Fu_R\|_{L_{\mu}^2(\cD;\cHt)}^2\rc^2 \\ &\qquad \leq
   C_F^2 c_{\nu}^{-1}r^2 \bE| f_{\nu}(t_k)^{-1/2} \|F_{t_k}u_R\|_{\cHt} |^2 \\ &\qquad=C_F^2
    c_{\nu}^{-1} r^2 \|Fu_R\|_{L_{\mu}^2(\cD;\cHt)}^2 \leq
    C_F^2 c_{\nu}^{-1} c_F^2 r^4 \eqqcolon \Sigma^2.
\end{align*}
By the Bernstein inequality (Lemma \ref{lem-thm-samp-grip-bernstein}) applied to $\varepsilon m$ in place of $\varepsilon$, we obtain 
\begin{align*}
    &\bP\lc \|QAu_R\|^2 \geq \|Fu_R\|_{L_{\mu}^2(\cD;\cHt)}^2 + \varepsilon \rc\\  &\qquad =
    \bP\lc \sum_{k=1}^m f_{\nu}(t_k)^{-1}\|F_{t_k}u_R\|_{\cHt}^2 \geq m\|Fu_R\|_{L_{\mu}^2(\cD;\cHt)}^2 + m\varepsilon \rc \\ &\qquad \leq
    \exp\lc - \frac{1}{C} \frac{m^2\varepsilon^2}{m C_F^2 c_{\nu}^{-1} c_F^2 r^4+m C_F^2 c_{\nu}^{-1} r^4 c_F^2+ C_F^2 c_{\nu}^{-1}r^2 m\varepsilon} \rc \\ &\qquad \leq
    \exp\lc -\frac{1}{C C_F^2 c_{\nu}^{-1}} \frac{m\varepsilon^2}{c_F^2 r^4+ r^2 \varepsilon} \rc,
\end{align*}
where $C>0$ is an absolute constant.

Therefore, with probability exceeding $1-\gamma$, the following estimate holds:
\begin{align*}
    \|QAu_R\|_{\cHt^m}^2 \leq \|Fu_R\|_{L_{\mu}^2(\cD;\cHt)}^2 + \varepsilon,
\end{align*}
provided that
\begin{align*}
    m \geq \tilde{C} \frac{c_F^2 r^4 + r^2\varepsilon}{\varepsilon^2} \log(1/\gamma),
\end{align*}
where $\tilde{C}\coloneqq C C_F^2 c_{\nu}^{-1}$.

We now invert the relation between $m=m(\varepsilon)$ and $\varepsilon$, obtaining $\varepsilon=\varepsilon(m)$. Then, the inequality above will be trivially satisfied for $m$ as in \eqref{eq:prop-trunc-sample-compl} and $\varepsilon=\varepsilon(m)$. We have that $\varepsilon=\varepsilon(m)$ is given by
\begin{align*}
    \varepsilon = \frac{\tilde{C}\log(1/\gamma)r^2+\sqrt{\tilde{C}^2\log^2(1/\gamma) r^4+4\tilde{C}\log(1/\gamma)c_F^2 r^4 m}}{2m}.
\end{align*}
Therefore, we obtain that, with probability exceeding $1-\gamma$, 
\begin{align*}
    \|QAu_R\|_{\cHt^m}^2 &\leq \|Fu_R\|_{L_{\mu}^2(\cD;\cHt)}^2 + \varepsilon \\ &\leq
    c_F^2 r^2 + \frac{\tilde{C}\log(1/\gamma)}{m} r^2 + \sqrt{ \frac{\tilde{C}\log(1/\gamma)}{m} } c_F r^2 \\ &\leq
     \lc c_F^2 + \frac{C C_F^2}{\|G_1^{-1}\|^2 \|G_2^{-1}\|^2 \|G_2\|^2} + \frac{\sqrt{C C_F^2} c_F}{\|G_1^{-1}\| \|G_2^{-1}\| \|G_2\|} \rc c_{\nu}^{-1} r^2.
\end{align*}
We now distinguish between the two cases in the statement of the Proposition. If $G_1=G_2\eqqcolon G$, then
\begin{align*}
    \|G^{-1}\|^2 \|QAu_R\|_{\cHt^m}^2 &\leq
    \lc c_F^2 \|G^{-1}\|^2 + \frac{C C_F^2}{\|G^{-1}\|^2\|G\|^2} + \frac{\sqrt{C C_F^2} c_F}{ \|G\|} \rc c_{\nu}^{-1} r^2\\ &\leq
    \lc c_F^2 \|G^{-1}\|^2 + C C_F^2 + c_F\|G^{-1}\|\sqrt{C C_F^2} \rc
    c_{\nu}^{-1} r^2 \\ &\leq
    C_1^{2}\lc c_F^2 \|G^{-1}\|^2 + 1 \rc c_{\nu}^{-1} r^2,
\end{align*}
with $C_1$ depending only on $C_F$.

On the other hand, if $\|G_1^{-1}\|\geq c$, we get that
\begin{align*}
    \|G_2^{-1}\|^2 &\|QAu_R\|_{\cHt^m}^2 \leq
    \lc c_F^2 \|G_2^{-1}\|^2 + \frac{C  C_F^2}{\|G_1^{-1}\|^2\|G_2\|^2} + \frac{\sqrt{C  C_F^2} c_F}{\|G_1^{-1}\| \|G_2\|}\|G_2^{-1}\| \rc c_{\nu}^{-1} r^2\\ &\leq
    \lc c_F^2 \|G_2^{-1}\|^2 + C C_F^2 c^{-2} \|G_2^{-1}\|^2 + c_F\sqrt{C  C_F^2} c^{-1} \|G_2^{-1}\|^2 \rc
    c_{\nu}^{-1} r^2 \\ &\leq
    C_1^{2}\lc c_F^2 \|G_2^{-1}\|^2 + 1 \rc c_{\nu}^{-1} r^2,
\end{align*}
with $C_1$ depending only on $C_F$ and  $c$.
\end{proof}

The proof of Theorem~\ref{thm-main-result} follows by a concatenation of the results already proved in this section.

\begin{proof}[Proof of Theorem~\ref{thm-main-result}]
We have that $m$ satisfies \eqref{eq-thm-samp-grip-01} in Theorem~\ref{thm-samp-grip} for
\begin{align*}
    \lambda = C_0 \|G^{-1}\|^2\|G\|s.
\end{align*}
As a consequence of Theorem~\ref{thm-samp-grip} we have that $QA$, where $Q$ is defined as in \eqref{eq:q-definition}, satisfies the g-RIP with respect to $(G,\omega,\delta,\lambda)$, with $\delta=1/2$. Using Theorem~\ref{thm-grip-rnsp} (with $\delta = 1/2$ and $\rho' = 1/3 $), we have that $QA$ satisfies the RNSP with respect to $(\omega,\rho,\kappa,s)$ with $\rho=1/2$ and $\kappa = \frac{3}{2\sqrt{1-\delta*}}\|G^{-1}\| = \frac{3}{\sqrt 2} \|G^{-1}\|$. We have that $C_0 = \frac{5}{\rho'^2}(1+\delta)/(1-\delta) = 135$.

Notice that $x^\dagger \in \ell^2(\Gamma)$, while the g-RIP property of $QA$ holds on the finite dimensional subspace $\ell^2(\Lambda) \subset \ell^2(\Gamma)$.
We can split the error term as follows:
\begin{align*}
    \|x^{\dagger}-\widehat{x}\|_2 \leq \|P_{\Lambda}x^{\dagger}-\widehat{x}\|_2 + \|P_{\Lambda}^{\perp}x^{\dagger}\|_2 \leq
    \|P_{\Lambda}x^{\dagger}-\widehat{x}\|_2 + r.
\end{align*}
We now estimate the first term. With reference to the $\ell^2$ bounds in Theorem~\ref{thm-rnsp-rec} applied to $P_{\Lambda}x^{\dagger}$ in place of $x^{\dagger}$ we get the following estimate:
\begin{align*}
    \|P_{\Lambda}x^{\dagger}-\widehat{x}\|_{2} &\leq
    \frac{c_1}{\sqrt{s}}\lc \|\widehat{x}\|_{1,\omega} - \|P_{\Lambda} x^{\dagger}\|_{1,\omega} + 2\sigma_s(P_{\Lambda}x^{\dagger})_{1,\omega} \rc +
    c_2\|G^{-1}\|
    \|QA(P_{\Lambda}x^{\dagger}-\widehat{x})\|_{\cH_2^m}
\end{align*}
for absolute constants $c_1,c_2>0$. Notice that $P_{\Lambda}x^{\dagger}$ satisfies the constraint of problem \eqref{eq-min-probl}; indeed, we can exploit Assumption~\ref{ass-prob-low-est} and Proposition \ref{prop:truncation-error} applied in the case $G_1=G_2=G$ to conclude that the following bound holds with probability exceeding $1-\gamma$:
\begin{align*}
   \left(\frac{1}{m}\sum_{k=1}^m \|F_{t_k}\Phi^*\iota_{\Lambda} x-y_k\|_{\cHt}^2\right)^{\frac12} &=\|AP_{\Lambda}x^{\dagger}-y\|_{\cH_2^m} \\ &\leq
    \|Ax^{\dagger}-y\|_{\cH_2^m}+
    \|AP_{\Lambda}^{\perp}x^{\dagger}\|_{\cH_2^m} \\ &\leq
    \beta+\|Q^{-1}\| \|QAP_{\Lambda}^{\perp}x^{\dagger}\|_{\cH_2^m} \\ &\leq
    \beta+C_3\|G^{-1}\|^{-1} r.
\end{align*}
Using the fact that $\widehat{x}$ is a minimizer of \eqref{eq-min-probl}, we get that
\begin{equation}
\label{eq:proof_thm35_split}
    \|P_{\Lambda}x^{\dagger}-\widehat{x}\|_{2} \leq
    \frac{2c_1}{\sqrt{s}}\sigma_s(P_{\Lambda}x^{\dagger})_{1,\omega} +
    c_2\|G^{-1}\|
    \|QA(P_{\Lambda}x^{\dagger}-\widehat{x})\|_{\cH_2^m}.
\end{equation}
We can split the last term on the RHS of \eqref{eq:proof_thm35_split} as follows:
\begin{align*}
    \|QA(P_{\Lambda}x^{\dagger}-\widehat{x})\|_{\cH_2^m} \leq
    \|QAx^{\dagger} - Qy\|_{\cH_2^m} +
    \|QA\widehat{x} - Qy\|_{\cH_2^m}+ 
    \|QAP_{\Lambda}^{\perp}x^{\dagger}\|_{\cH_2^m}.
\end{align*} 
The first term satisfies
\begin{align*}
    \|QAx^{\dagger}-Qy\|_{\cH_2^m} \leq c_{\nu}^{-1/2}\beta
\end{align*}
by the assumption on the noise level and by Assumption~\ref{ass-prob-low-est}. The second term satisfies
\begin{align*}
    \|QA\widehat{x}-Qy\|_{\cH_2^m} \leq c_{\nu}^{-1/2}(\beta+C_3\|G^{-1}\|^{-1}r)
\end{align*}
because $\widehat{x}$ satisfies the constraint of problem \eqref{eq-min-probl}. Finally, the third term satisfies
\begin{align*}
    \|QAP_{\Lambda}^{\perp}x^{\dagger}\|_{\cH_2^m} \leq C_3 \|G^{-1}\|^{-1} r \leq c_{\nu}^{-1/2} C_3 \|G^{-1}\|^{-1} r
\end{align*} 
with probability exceeding $1-\gamma$ by Proposition \ref{prop:truncation-error} applied in the case $G_1=G_2=G$.

These bounds together imply the following bound on the last term in the RHS of \eqref{eq:proof_thm35_split}: 
\begin{align*}
    \|G^{-1}\| \|QA(P_{\Lambda}x^{\dagger}-\widehat{x})\|_{\cH_2^m} \leq
    2c_{\nu}^{-1/2} (\|G^{-1}\|\beta+C_3r).
\end{align*} 
The bound given by Proposition \ref{prop:truncation-error} holds together with $QA$ satisfying the g-RIP with probability exceeding $1-2\gamma$. This concludes the proof in the $\ell^2$ case. The $\ell^1_{\omega}$ case is dealt analogously.
\end{proof}

\subsection{Quantitative bounds on \texorpdfstring{$\|G^{-1}\|$}{|G-1|}}
We prove the results stated in Section~\ref{subsec-exp-est}.

\begin{proof}[Proof of Proposition~\ref{prop-cond-quasi-diag}]
 The assumptions  imply that
\begin{equation}
    \|Fu\|_{\cHd}^2 \asymp \|u\|_{H^{-b}}^2 \asymp \sum_{(j,n)\in\Gamma} 2^{-2bj}|\langle u,\phi_{j,n} \rangle|^2.
\end{equation}
Consider now $x\in\ell^2(\Gamma)$ and set $u\coloneqq \Phi^* x$. Since $(\phi_{j,n})_{(j,n)\in\Gamma}$ is an orthonormal basis by assumption, we have $x_{j,n}=\langle u,\phi_{j,n} \rangle_{L^2}$. Then
\begin{equation}
    \|F\Phi^* x\|_{\cHd}^2 = \|Fu\|_{\cHd}^2 \asymp
    \sum_{(j,n)\in\Gamma} 2^{-2bj}|\langle u,\phi_{j,n} \rangle|^2 =
    \sum_{(j,n)\in\Gamma} 2^{-2bj}|x_{j,n}|^2,
\end{equation} hence $F$ satisfies the quasi-diagonalization property as claimed. 
\end{proof}

\begin{proposition}
\label{prop-quasi-diag-est}
Suppose that an operator $F$ satisfies the quasi-diagonalization property \eqref{eq-quasi-diag-gamma} with respect to $(\Phi,b)$ $b\geq 0$. Let $G\coloneqq \sqrt{P_{\Lambda}\Phi F^* F \Phi^*\iota_{\Lambda}}$ and $Z\coloneqq\diag(z_{j,n})_{(j,n)\in\Lambda}$ be a diagonal matrix, with $z_{j,n}>0$ for every $(j,n)\in\Lambda$. Then the following estimate holds:
\begin{equation}
    \|(GZ)^{-1}\|^2 \asymp
    \max_{(j,n)\in\Lambda} (z_{j,n}^{-2} 2^{2bj}),
\end{equation}
where the implicit constants depend on the quasi-diagonalization constants in \eqref{eq-quasi-diag-gamma}.
In particular, if $z_{j,n}\equiv 1$ and $\Lambda=\Lambda_{j_0}$, then
\begin{equation}
    \|G^{-1}\|^2 \leq C2^{2bj_0}.
\end{equation}
\end{proposition}
\begin{proof}
Set $W\coloneqq \diag(2^{bj})_{(j,n)\in\Lambda}$. By \eqref{eq-quasi-diag-gamma} we have that
\begin{equation}
    \|GZx\|_2^2 =\|F\Phi^*\iota_\Lambda Z x\|_{\cHd}^2 \geq c\|W^{-1}Zx\|_2^2 \geq c\min_{(j,n)\in\Lambda} (z_{j,n}^{2} 2^{-2bj})\|x\|_2^2.
\end{equation}
Combining these two estimates yields
\begin{align*}
    \|x\|_2^2 \leq \frac{1}{c} \max_{(j,n)\in \Lambda} (z_{j,n}^{-2} 2^{2bj}) \|GZx\|_2^2.
\end{align*}
This proves the upper bound on $\|(GZ)^{-1}\|^2$.

On the other hand, let $(j',n')\in\Lambda$ be such that
\begin{align*}
    z_{j',n'}^2 2^{-2bj'} = \min_{(j,n)\in\Lambda} (z^2_{j,n}2^{-2bj}).
\end{align*}
Then we have that
\begin{align*}
    \|GZe_{j',n'}\|_2^2 \leq C\|W^{-1}Ze_{j',n'}\|_2^2 = C\min_{(j,n)\in\Lambda} (z_{j,n}^2 2^{-2bj}).
\end{align*}
Let $f_{j',n'}\coloneqq \frac{GZe_{j',n'}}{\sqrt{C}\min_{(j,n)\in\Lambda} (z_{j,n}2^{-bj})}$. The previous bound implies that $\|f_{j',n'}\|_2\leq 1$. Therefore, we have that
\begin{align*}
    \|(GZ)^{-1}\|^2 =
    \sup_{\|y\|_2\leq 1} \|(GZ)^{-1}y\|^2 \geq
    \|(GZ)^{-1}f_{j',n'}\|_2^2 =
    \frac{1}{C\min_{(j,n)\in\Lambda}(z_{j,n}^2 2^{-2bj})}.
\end{align*}
This proves the lower bound on $\|(GZ)^{-1}\|^2$.
\end{proof}

\begin{proof}[Proof of Theorem~\ref{thm:interpolation}]
Let $D\coloneqq\diag(d_{j,n})_{(j,n)\in\Lambda_{j_0}}$ and let $\tilde{D}$ be the operator acting on $\Span(\phi_{j,n})_{(j,n)\in\Gamma}$ by $\tilde{D}\phi_{j,n}\coloneqq d_{j,n}\phi_{j,n}$ for $(j,n)\in\Lambda_{j_0}$,  $\tilde{D}\phi_{j,n}=\phi_{j,n}$ otherwise, so that $\Phi^* D = \tilde{D} \Phi^*$. Moreover, we extend $W$, and therefore $W^{\zeta}$, to an operator on $\ell^2(\Gamma)$, setting $W=\diag(w_{j,n})_{j,n}$, where
\begin{align*}
    w_{j,n} =
    \begin{cases}
        2^{bj} & \text{if $(j,n)\in\Lambda_{j_0}$},\\
        1 & \text{otherwise}.
    \end{cases}
\end{align*}
Similarly, set $\tilde W^\zeta = \Phi^* W^\zeta \Phi$.

Consider the measurement operators given by $(F_t\tilde{D})$. The companion forward map is given by $F\tilde{D}$, while the truncated matrix (i.e., playing the role of $G$) is given by $\tilde{G}\coloneqq \sqrt{D^*G^*GD}$ and the corresponding sampling matrix is $AD$. Furthermore, $\|(\tilde G) ^{-1}\| = \|(GD)^{-1}\|$.

The following bounds hold by Proposition~\ref{prop-quasi-diag-est} and by the quasi-diagonalization property:
\begin{equation}\label{eq:upandlow}
    \|(GD)^{-1}\|^2 \asymp \max_{(j,n)\in\Lambda_{j_0}} (d_{j,n}^{-2} 2^{2bj}),\quad
    \|(GW^{\zeta})^{-1}\|^2\asymp 2^{2(1-\zeta)bj_0},\quad
    \|GW^{\zeta}\| \lesssim 1,
\end{equation}
where the implicit constants depend on the quasi-diagonalization constants. Let $\lambda$ be defined by
\begin{align*}
    \lambda \coloneqq \tilde{C_0} \|GW^{\zeta}\|^2 \|(GW^{\zeta})^{-1}\|^2 s,
\end{align*}
where $\tilde{C_0}=135$ is the constant appearing in the expression for $\lambda$ in Theorem~\ref{thm-grip-rnsp} with $\delta=1/2$ and $\rho'=1/3$.

Using these bounds, we deduce that, if $C_0$ is large enough, the sample complexity $m$ in \eqref{eq:thm-interp-samplecomplexity} satisfies \eqref{eq-thm-samp-grip-01}, namely,
\begin{equation}
\label{eq:thm-int-sampcomp1}
    m \geq \frac{C}{4}\tau\max\{\log^3{\tau}\log{M},\log(1/\gamma)\}
\end{equation}
with $\tau= B^2\|(GD)^{-1}\|^2\lambda \geq 3$. By Theorem~\ref{thm-samp-grip}, we conclude that $QAD$ satisfies the g-RIP with respect to $(GD,\omega,1/2,\lambda)$ with overwhelming probability. The diagonal invariance of the g-RIP proved in Remark~\ref{rem-diag-inv} shows that the latter result implies that $QA$ satisfies the g-RIP with respect to $(G,\omega,1/2,\lambda)$.

By virtue of the diagonal invariance of the g-RIP illustrated in Remark \ref{rem-diag-inv}, $QAW^{\zeta}$ equivalently satisfies the g-RIP with respect to $(GW^{\zeta},\omega,1/2,\lambda)$. Theorem~\ref{thm-grip-rnsp} now implies that $QAW^{\zeta}$ satisfies the RNSP with respect to $(\omega,\rho,\kappa,s)$ with $\rho=1/2$ and $\kappa=\frac{3}{\sqrt{2}}\|(GW^{\zeta})^{-1}\|\lesssim 2^{(1-\zeta)bj_0}$.

The following identity holds:
\begin{align*}
    y = AW^{\zeta}z^{\dagger}+\varepsilon,
\end{align*}
where $z^{\dagger}\coloneqq W^{-\zeta}x^{\dagger}$. Notice that $\|P_{\Lambda_{j_0}}^{\perp}z^{\dagger}\|_2 = \|W^{-\zeta}P_{\Lambda_{j_0}}^{\perp}x^{\dagger}\|_2 \leq 2^{-\zeta bj_0}r\eqqcolon r'$. Let $\widehat{z}$ be a solution of the problem
\begin{equation}
    \min_{z\in\ell^2(\Lambda_{j_0})} \|z\|_{1,\omega}\colon\quad \|AW^{\zeta}z-y\|_{\cH_2^m}^2 \leq \lc \beta+ C_3 2^{(\zeta-1)bj_0} r'\rc^2 = \lc \beta+ C_3 2^{-bj_0} r\rc^2,
\end{equation}
where $C_3$ is a constant to be assigned later. We also introduce the notation $\widehat{x}\coloneqq W^{\zeta}\widehat{z}$. The change of variable $x=W^{\zeta}z$ shows that this problem is equivalent to \eqref{eq-dir-min-probl}.

We now proceed as in the proof of Theorem~\ref{thm-main-result}. Since $$\|W^{-\zeta}x^{\dagger}-W^{-\zeta}\widehat{x}\|_2=\|z^{\dagger}-\widehat{z}\|_2\leq \|P_{\Lambda_{j_0}}z^{\dagger}-\widehat{z}\|_2 + r',$$ it is enough to obtain a bound for $\|P_{\Lambda_{j_0}}z^{\dagger}-\widehat{z}\|_2$.

The companion forward map associated with the measurement operators $(F_t \tilde W^{\zeta})$ is given by $F \tilde W^{\zeta}$. Using the quasi-diagonalization property \eqref{eq-quasi-diag-gamma}, we deduce that $\|F\tilde W^{\zeta}\|\leq C$, where $C$ is the constant in the quasi-diagonalization upper bound; therefore, an upper bound $C_{F \tilde W^{\zeta}}$ associated with this forward map is given by $C$. Using Proposition~\ref{prop:truncation-error} applied in the case $G_1=GD$ and $G_2=GW^{\zeta}$ (indeed notice that, by \eqref{eq:upandlow}, $\|(GD)^{-1}\|$ is bounded from below by a constant depending on the quasi-diagonalization bounds, since $d_{j,n}\leq 2^{bj}$), we get that, with probability exceeding $1-\gamma$, the following bound holds:
\begin{align*}
    \|QAW^{\zeta}P^\perp_{\Lambda_{j_0}}z^{\dagger}\|_{\cH_2^m} &\leq
    C c_{\nu}^{-1/2} \|(GW^{\zeta})^{-1}\|^{-1} \lc \|F\Phi^*W^{\zeta}\iota_{\Gamma\setminus\Lambda_{j_0}}\| \|(GW^{\zeta})^{-1}\|+1 \rc r' \\ &=
    C c_{\nu}^{-1/2} \lc \|F\Phi^*W^{\zeta}\iota_{\Gamma\setminus\Lambda_{j_0}}\|+\|(GW^{\zeta})^{-1}\|^{-1} \rc r' \\ &\leq
    C c_{\nu}^{-1/2} 2^{(\zeta-1)bj_0} r' = C c_{\nu}^{-1/2} 2^{-bj_0}r,
\end{align*}
where $C>0$ depends only on the quasi-diagonalization constants; we have also used \eqref{eq:upandlow} to bound $\|(GW^{\zeta})^{-1}\|^{-1}\lesssim 2^{(\zeta-1)bj_0} $ and the quasi-diagonalization property to bound $\|F\Phi^*W^{\zeta}\iota_{\Gamma\setminus\Lambda_{j_0}}\|\lesssim 2^{(\zeta-1)bj_0} $. We set $C_3\coloneqq C c_{\nu}^{-1/2}$. Arguing as in the proof of Theorem~\ref{thm-main-result}, we can infer that $z^{\dagger}$ satisfies the constraints of the minimization problem; using the distance bounds from Theorem \ref{thm-rnsp-rec}, we conclude.
\end{proof}

\begin{proof}[Proof of Corollary \ref{cor:main-result-D}]
It is enough to apply Theorem \ref{thm:interpolation} in the case $\zeta=0$.
\end{proof}

\begin{proof}[Proof of Corollary \ref{cor:dir-recovery}]
It suffices to apply Theorem \ref{thm:interpolation} in the case $\zeta=1$ and to notice that, by the quasi-diagonalization property \eqref{eq-quasi-diag-gamma},
\begin{equation}
    \|F\Phi^* x^{\dagger}-F\Phi^*\widehat{x}\|_{L_{\mu}^2(\cD;\cHt)} \asymp
    \|W^{-1}x^{\dagger}-W^{-1}\widehat{x}\|_2. \qedhere    
\end{equation}
\end{proof}

\begin{proof}[Proof of Theorem \ref{thm:exp-est}]
We first introduce the convenient notation \[ \tilde{m}\coloneqq m/(\tilde{C_0} B\log^3{\tau}\log{M})\gtrsim m/(\tilde{C_0} B cj_0\log^3{\tau}),\] where $\tilde{C_0}=C_0\log(1/\gamma)$ and $C_0,B$ are the constants appearing in Theorem \ref{cor:main-result-D} and $\tau\coloneqq 2^{bj_0}s$ for some $s\geq 3$. Notice that, by the choice of $j_0$, we have $j_0\asymp \log(1/\beta)$ and 
\begin{align*}
    \log{\tau}\leq2bj_0+\log{s}\leq bj_0+\log{M}\asymp j_0 \asymp \log(1/\beta).
\end{align*}
We conclude that
\begin{align}
\label{eq:rel-m-mtilde}
    \tilde{m} \gtrsim m/\log^4(1/\beta).
\end{align}
Suppose that $\beta$ is sufficiently small so that $\log^3{\tau}\log{M}\geq \log(1/\gamma)$ -- recall that $j_0=j_0(\beta)$.
By Theorem \ref{cor:main-result-D}, if
\begin{align*}
    \tilde{m} \geq 2^{2(b-d)j_0} 2^{2bj_0} s \geq \max_{(j,n)\in\Lambda_{j_0}}(d_{j,n}^{-2} 2^{2bj}) 2^{2bj_0} s
\end{align*}
then, with probability exceeding $1-\gamma$, the following error estimate holds:
\begin{align*}
    \|x^{\dagger}-\widehat{x}\|_{2} \leq C_1\lc \frac{\sigma_s(x^{\dagger})_1}{\sqrt{s}}+2^{bj_0}\beta + r \rc.
\end{align*}
Suppose that $2^{-bj_0}\asymp \beta^{\zeta}$ for some $\zeta>0$. Therefore 
$\tilde{m} = \beta^{-2\frac{(2b-d)}{b}\zeta}s$
and
\begin{align*}
        \|x^{\dagger}-\widehat{x}\|_{2} \leq C_1\lc s^{-p}+\beta^{1-\zeta}+\beta^{\frac{a}{b}\zeta} \rc.
\end{align*}
Imposing $1-\zeta=a\zeta/b$, we get that $\zeta=b/(a+b)$, which  yields $2^{-bj_0}\asymp\beta^{\frac{b}{a+b}}$.

Combining the expressions for $s$ and $\zeta$ with the estimates, we get 
\begin{align*}
     \|x^{\dagger}-\widehat{x}\|_{2} \leq C_1\lc \frac{\beta^{-2\frac{(2b-d)}{a+b} p}}{\tilde{m}^p}+\beta^{\frac{a}{a+b}} \rc.
\end{align*}
Using \eqref{eq:rel-m-mtilde}, the first estimate in the claim follows.

Moreover, if $\tilde{m}=\beta^{-2\frac{(2b-d)}{a+b}-\frac{a}{p(a+b)}}$, then $\tilde{m}^p=\beta^{-2\frac{(2b-d)}{a+b}p-\frac{a}{a+b}}$ and therefore
\begin{align*}
    \frac{\beta^{-2\frac{(2b-d)}{a+b} p}}{\tilde{m}^p} = \beta^{\frac{a}{a+b}}.
\end{align*}
The second estimate follows.

The last estimate is obtained by inverting the relation $\tilde{m}=\tilde{m}(\beta)$, using again \eqref{eq:rel-m-mtilde} and observing that $\log{m}\asymp\log{1/\beta}$.
\end{proof}

\subsection{The sparse Radon transform}
\label{sec:proof-radon}
We recall some notation from Section~\ref{sec-radon}.

The space of signals for both problems is given by $\cH_1 = \overline{\Span(\phi_{j,n})_{(j,n)\in\Gamma}}$, while $\cHt = L^2(\bR)$ is the codomain for the Radon transform at a fixed angle $\theta\in[0,2\pi)$ and $\cHt = L^2(-\pi/2,\pi/2)$ is the codomain for the fan beam transform from the angle $\theta$. Let $K$ be the compact set defined by
\begin{align*}
    K\coloneqq \overline{ \bigcup_{(j,n)\in\Gamma} \supp{\phi_{j,n}} },
\end{align*}
so that $\cH_1\subset L^2(K)$.

A suitable measurement space for this model is $(\cD,\mu) = ([0,2\pi),\diff{\theta})$, where $\diff{\theta}$ is the uniform probability measure on $[0,2\pi)$. For every $\theta\in\bS^1$, $F_\theta\colon \cH_1\subset L^2(K)\rightarrow L^2(\bR)$ is given by either the Radon transform at a fixed fixed angle $F_\theta = \cR_\theta$ or the fan beam transform from a fixed angle $F_{\theta}=\cD_{\theta}$. The forward map $F\colon \cH_1\subset L^2(K)\rightarrow L^2([0,2\pi)\times\bR)$ is represented by either the Radon transform $F = \cR$ or the fan beam transform $F=\cD$. Notice that the codomain of the Radon transform $L^2([0,2\pi)\times\bR)$ fits into our framework via the canonical identification $L^2([0,2\pi)\times\bR)\cong L^2\lc[0,2\pi),L^2(\bR)\rc$, and the same applies to the fan beam transform.

\begin{lemma}
\label{thm-radon-quasidiag}
The following inequalities hold for some constants $c_1,C_1>0$ depending on $K$:
\begin{equation}
    c_1\|u\|_{H^{-1/2}} \leq \|\cR u\|_{L^2} \leq C_1\|u\|_{H^{-1/2}},\quad u\in L^2(K).
\end{equation}
\end{lemma}
\begin{proof}
    The proof is analogous to \cite[Theorem 5.1]{natterer}, where the inqualities are proved for $u\in L^2(\cB_1)$.
\end{proof}

\subsubsection*{Quasi-diagonalization}  Theorem~\ref{thm-radon-quasidiag} implies that the Radon transform satisfies condition \eqref{eq-suff-quasi-diag2} with $b=1/2$. On the other hand, if the wavelets are sufficiently regular as detailed in Remark~\ref{rem-regwav-quasidiag}, the Littlewood-Paley property \eqref{eq-suff-quasi-diag1} holds with $b=1/2$. Then, Proposition~\ref{prop-cond-quasi-diag} yields the following quasi-diagonalization property:
\begin{equation}
\label{eq:radon-quasi-diag}
    c_1 \sum_{(j,n)\in\Gamma} 2^{-j} |x_{j,n}|^2 \leq \|\cR\Phi^* x\|_{L^2}^2 \leq C_1 \sum_{(j,n)\in\Gamma} 2^{-j} |x_{j,n}|^2,\quad
    x\in\ell^2(\Gamma).
\end{equation}

We now prove a quasi-diagonalization property for the fan beam transform using the following identity (see, for instance, \cite[Section III.3]{natterer} or \cite[Section 16.3]{scherzer_book}):
\begin{equation}
\label{eq:radon-divergent}
    \cD{u}(\theta,\alpha) = \cR{u}(\theta+\alpha-\pi/2,\rho\sin{\alpha}).
\end{equation}
We have that
\begin{align*}
    \|\cD u\|_{L^2}^2 &= \int_{\bS^1} \int_{-\pi/2}^{\pi/2} |\cR{u}(\theta+\alpha-\pi/2,\rho\sin{\alpha})|^2 \diff{\alpha}\diff{\theta} \\ &=
    \int_{\bS^1} \int_{-\pi/2}^{\pi/2} |\cR{u}(\theta,\rho\sin{\alpha})|^2 \diff{\alpha}\diff{\theta} \\ &=
    \int_{\bS^1} \int_{-r}^r |\cR(\theta,s)|^2 \frac{1}{\sqrt{\rho^2-s^2}} \diff{s}\diff{\theta},
\end{align*}
which implies that
\begin{equation}
\label{eq:radon-quasi-diag-bis}
    \rho^{-1/2} \|\cR{u}\|_{L^2} \leq \|\cD u\|_{L^2} \leq (\rho^2-d^2)^{-1/4} \|\cR{u}\|_{L^2}.
\end{equation}
This, together with Lemma \ref{thm-radon-quasidiag} and \eqref{eq:radon-quasi-diag}, implies that $\cD$ satisfies the quasi-diagonalization property with $b=1/2$, with constants depending on the chosen wavelet basis and on $d$ and $\rho$.

\subsubsection*{Coherence bounds}
We notice that, by construction, for every line $\ell\subset\bR^2$, $\supp{\phi_{j,n}}\cap\ell$ is contained in a segment of length $C2^{-j}$, where $C$ depends only on the chosen wavelet basis. We then have that
\begin{align*}
    \|\cR_{\theta}\phi_{j,n}\|_{L^2(\bR)}^2 &= \int_{\bR} \bigg| \int_{\theta^{\perp}} \phi_{j,n}(y+s e_\theta) \diff{y} \bigg|^2 \diff{s} \\ &\leq
    C2^{-j} \int_{\bR} \int_{\theta^{\perp}} |\phi_{j,n}(y+s e_\theta)|^2 \diff{y} \diff{s} \\ &=
    C2^{-j} \|\phi_{j,n}\|_{L^2(\bR^2)}^2 = C2^{-j}.
\end{align*}
In the same way, we have that
\begin{align*}
    \|\cD_{\theta}\phi_{j,n}\|_{L^2}^2 &=
    \int_{-\pi/2}^{\pi/2} \bigg| \int_{\bR} \phi_{j,n}(\rho e_{\theta}+te_{\theta+\alpha})\diff{t} \bigg|^2 \diff{\alpha} \\ &\leq
    C2^{-j} \int_{-\pi/2}^{\pi/2} \int_{\bR} |\phi_{j,n}(\rho e_{\theta}+te_{\theta+\alpha})|^2\, \diff{t} \diff{\alpha} \\ &=
    C2^{-j} \int_{-\pi/2}^{\pi/2} \int_{\bR} |\phi_{j,n}(\rho e_{\theta}+te_{\theta+\alpha})|^2\, t\, \frac{1}{t}\, \diff{t} \diff{\alpha} \\ &\leq
    \frac{C2^{-j}}{\rho-d} \int_{-\pi/2}^{\pi/2} \int_{\bR} |\phi_{j,n}(\rho e_{\theta}+te_{\theta+\alpha})|^2\, t\, \diff{t} \diff{\alpha} =
    \frac{C2^{-j}}{\rho-d}.
\end{align*}
Assumption \ref{ass-coherence-bound-D} is then satisfied in both cases by $d_{j,n}=2^{j/2}$ and a constant $B$, which, in the case of the fan beam transform, depends on $d$, $\rho$ and the chosen wavelet basis:
\begin{equation}
\label{eq:qd-radon-div}
 \|F_{\theta}\phi_{j,n}\|_{L^2} \leq \frac{B}{2^{j/2}},\quad
    \theta\in [0,2\pi),\ (j,n)\in\Gamma.
\end{equation}
In both cases, the quasi-diagonalization property, together with Proposition~\ref{prop-quasi-diag-est}, implies the following estimates:
\begin{equation}
    \|(GD)^{-1}\|^2 \leq C\max_{(j,n)\in\Lambda_{j_0}}(2^j2^{-j}) = C,
\end{equation}
and
\begin{equation}
    \|G^{-1}\|^2 \leq C2^{j_0}.
\end{equation}
Moreover, by \eqref{eq:radon-estimate-number-wavelets}, we have that $\log{M}\asymp 2j_0$.

\begin{proof}[Proof of Theorem \ref{thm:radon-main-thm}]
    It suffices to apply Corollary \ref{cor:main-result-D} to the Radon setting, using the quasi-diagonalization properties \eqref{eq:radon-quasi-diag} and \eqref{eq:radon-quasi-diag-bis} and the coherence bounds \eqref{eq:qd-radon-div}.
\end{proof}

\begin{proof}[Proof of Theorem \ref{thm:radon-dir-prob}]
    It suffices to apply Corollary \ref{cor:dir-recovery} to the Radon setting, using the quasi-diagonalization properties \eqref{eq:radon-quasi-diag} and \eqref{eq:radon-quasi-diag-bis} and the coherence bounds \eqref{eq:qd-radon-div}.
\end{proof}

\begin{proof}[Proof of Theorem \ref{thm:radon-exp-est}]
    It suffices to apply Theorem \ref{thm:exp-est} to the Radon setting, using the quasi-diagonalization properties \eqref{eq:radon-quasi-diag} and \eqref{eq:radon-quasi-diag-bis} and the coherence bounds \eqref{eq:qd-radon-div}.
\end{proof}

\begin{proof}[Proof of \eqref{eq:ex-cartoon-like-est}]
    We mainly follow the arguments used in \cite[Section 9.3.1]{Ma}. From the construction of the basis of compactly supported wavelets, we have that $\|\phi_{j,n}\|_{L^2}=1$ and $|\supp(\phi_{j,n})|=C2^{-2j}$. Therefore, we get that
    \begin{align*}
        |\langle u^{\dagger},\phi_{j,n} \rangle| \leq \|u^{\dagger}\|_{L^{\infty}} \|\phi_{j,n}\|_{L^1} \leq C2^{-j}.
    \end{align*}
    On the other hand, in view of Remark~\ref{rem-regwav-quasidiag} with $b=-2$, we have that $|\langle u^{\dagger},\phi_{j,n} \rangle|\lesssim 2^{-2j}$ if $\supp(\phi_{j,n})$ does not intersect the discontinuities of $u^\dagger$.

    Let $\Sigma_j^1$ be the set of indices of wavelets at scale $j$ whose support intersects the discontinuities of $u^\dagger$, and  $\Sigma_j^2$ be the set of indices of wavelets at scale $j$ whose support does not intersects the discontinuities of $u^\dagger$. It is possible to see, arguing as in the part on \textit{NonLinear Approximation of Piecewise Regular Images} 
 of \cite[Section 9.3.1]{Ma}, that $|\Sigma_j^1|\asymp 2^j$ and $|\Sigma_j^2|\asymp 2^{2j}$. Therefore we have that
    \begin{align*}
        \|P_{\Lambda_{j_0}}^{\perp}\Phi u^{\dagger}\|_2^2 &=
        \sum_{j>j_0} \sum_n |\langle u^{\dagger},\phi_{j,n} \rangle|^2 \\ &=
        \sum_{j>j_0} \lc \sum_{(j,n)\in\Sigma_j^1} |\langle u^{\dagger},\phi_{j,n} \rangle|^2 +
        \sum_{(j,n)\in\Sigma_j^2}|\langle u^{\dagger},\phi_{j,n} \rangle|^2 \rc \\ &\lesssim
        \sum_{j>j_0} (2^j 2^{-2j} + 2^{2j}2^{-4j}) \lesssim 2^{-j_0}.
    \end{align*}
    This implies that $\|P_{\Lambda_j}^{\perp}\Phi u^{\dagger}\|_2 \leq C 2^{-aj}$ is satisfied with $a=1/2$.

    From a similar argument (see formula $(9.63)$ in \cite{Ma}), it follows that $|(\Phi{u^{\dagger}})^{(i)}|\leq i^{-1}$, where $|(\Phi{u^{\dagger}})^{(\cdot)}|$ is the non-increasing rearrangement of $|\Phi{u^{\dagger}}|$. It is not possible to infer a $p$-compressibility estimate since $\Phi{u^{\dagger}}$ might not even be summable and therefore we cannot resort to Theorem \ref{thm:radon-exp-est}. Instead, we will exploit the fact that the sparsity error in the recovery estimates of Theorem \ref{thm:radon-main-thm} depends on the truncated signal $P_{\Lambda_{j_0}}\Phi{u}^{\dagger}$. More precisely, for a certain fixed sparsity $s\geq 1$, we have that
    \begin{align*}
        \sigma_s(P_{\Lambda_{j_0}}\Phi{u^{\dagger}})_1 \leq \sum_{i=s}^M |(\Phi{u^{\dagger}})^{(i)}| \leq
        \sum_{i=s}^M \frac{1}{i} \lesssim \log(M/s) \leq \log(M),
    \end{align*}
    where $M=|\Lambda_{j_0}|$. Recalling that $M\asymp 2^{2j_0}$ and choosing $j_0$ as a function of $\beta$ as in Theorem  \ref{thm:radon-exp-est}, we get that $\log(M)\lesssim \log(1/\beta)$. Therefore we conclude that
    \begin{align*}
        \sigma_s(P_{\Lambda_{j_0}}\Phi{u^{\dagger}})_1 \lesssim \log(1/\beta)s^{1/2-p}
    \end{align*}
    with $p=1/2$. The rest of the estimates are deduced in the exact same way as in the proof of Theorem \ref{thm:radon-exp-est}, which imply \eqref{eq:ex-cartoon-like-est}.
\end{proof}

\bibliography{refs}{}
\bibliographystyle{abbrv}

\appendix

\section{Proof of Theorem \ref{thm-samp-grip}}
\label{appendix:proofs}
We first recall a variant of the Bernstein-Talagrand inequality that will be needed later.
\begin{lemma}[\textbf{Bernstein inequality} {\cite[Theorem 8.42]{FR}}]
\label{lem-thm-samp-grip-bernstein}
Let $f_x\colon\bC^M\rightarrow\bR$, be a countable family of functions indexed by $x\in\cB_*$. Let $Y_1,\dots,Y_m$ be independent random variables such that the following conditions are satisfied for every $k=1,\dots,m$ and for $x\in\cB_*$:
\begin{equation}
    \bE f_x(Y_k) = 0;\quad
    |f_x(Y_k)|\leq K \text{ a.s.};\quad
    \bE|f_x(Y_k)|^2 \leq \Sigma^2.
\end{equation}
Consider the random variable
\begin{equation}
    Z \coloneqq \sup_{x\in\cB_*} \sum_{k=1}^m f_x(Y_k).
\end{equation}
Then, for every $\varepsilon>0$,
\begin{equation}
    \bP(Z\geq\bE Z+\varepsilon) \leq \exp\lc
    - \frac{\varepsilon^2}{2(m\Sigma^2+2K\bE Z)+2K\varepsilon/3} \rc.
\end{equation}
\end{lemma}

\subsection{The case $f_{\nu}= 1$} Let us first consider the case where $f_{\nu}= 1$, so that $\mu=\nu$ is a probability measure. Assumption \ref{ass-coherence-bound} thus becomes
\begin{equation}
\label{eq-thm-samp-grip-assmodified}
    \|F_t\phi_i\|_{\cHt} \leq B\omega_i,\quad t\in\cD,\ i\in\Gamma.
\end{equation}
Furthermore, we have $Q=I$.

Consider the following seminorm on $\bC^{M\times M}$:
\begin{equation}
    \|B\|_{G,\omega,\lambda} \coloneqq \sup_{x\in\cB} |\langle Bx,x \rangle|, 
\end{equation} where we set
\begin{equation}
    \cB \coloneqq \{x\in\bC^M\colon\ \|Gx\|_2\leq 1,\ \|x\|_{0,\omega}\leq\lambda\}.
\end{equation}
A standard argument allows us to characterize the smallest $\delta>0$ for which $A$ satisfies the g-RIP with respect to $(G,\omega,\delta,\lambda)$ as
\begin{equation}\label{lem-thm-samp-grip-delta}
    \delta_* \coloneqq \|A^*A - G^*G\|_{G,\omega,\lambda}.
\end{equation}
Moreover, let $$X_t \coloneqq F_t\Phi^*\iota_{\Lambda},\quad t\in\cD.$$ Elementary computations show that
\begin{equation} \label{lem-thm-samp-grip-astara}
    A^*A = \frac{1}{m} \sum_{k=1}^m X_{t_k}^* X_{t_k},
\end{equation}
\begin{equation}\label{lem-thm-samp-grip-xstarx}
    G^*G = \bE X_t^* X_t. 
\end{equation}

The goal is now to provide a bound for  $\bE \delta_* = \bE \| A^*A - G^*G \|_{G, \omega,\lambda}$. To this aim, let us introduce a Rademacher vector $(\varepsilon_1,\dots,\varepsilon_m)$, independent of the samples $(t_1,\dots,t_m)$, and consider the following random variable, for $u\in\cB$:
\begin{equation}
    Z_u \coloneqq \sum_{k=1}^m \varepsilon_k \|X_{t_k}u\|_{\cHt}^2.
\end{equation}
It is straightforward to check that $\{Z_u\colon\ u\in\cB\}$ is a centered, symmetric, sub-Gaussian stochastic process \cite[Section 8.6]{FR} with respect to the following pseudo-metric:
\begin{equation}
    d(u,v) \coloneqq \lc \sum_{k=1}^m ( \|X_{t_k}u\|_{\cHt}^2 - \|X_{t_k}v\|_{\cHt}^2 )^2 \rc^{1/2}.
\end{equation} We are thus in the position to use Dudley's inequality.

\begin{proposition}\label{prop bound bE}
The following estimate holds:
\begin{equation}
    \bE\delta_* \leq \frac{16\sqrt{2}}{m} \bE\int_0^{\infty} \sqrt{\log\cN(\cB,d,y)}\,\diff{y}.
\end{equation}
\end{proposition}
\begin{proof}
In light of \eqref{lem-thm-samp-grip-delta}, \eqref{lem-thm-samp-grip-astara} and \eqref{lem-thm-samp-grip-xstarx}, we have 
\begin{equation} \bE\delta_* = \bE\bigg\|\frac{1}{m}\sum_{k=1}^m (X_{t_k}^* X_{t_k}-\bE X_{t_k}^*X_{t_k})\bigg\|_{G,\omega,\lambda}.
\end{equation} 
Using a standard symmetrization trick (see, for instance, \cite[Lemma 8.4]{FR}), we get
\begin{equation} \bE\delta_* \le \frac{2}{m}\bE\bigg\| \sum_{k=1}^m \varepsilon_k X_{t_k}^*X_{t_k} \bigg\|_{G,\omega,\lambda} =
    \frac{2}{m}\bE\sup_{u\in\cB} \bigg| \sum_{k=1}^m \varepsilon_k \|X_{t_k}u\|_{\cHt}^2 \bigg|. 
\end{equation} 
The desired conclusion directly follows from Dudley's inequality.
\end{proof}
In order to get a manageable estimate for the covering numbers, we first slightly modify the underlying metric. Before, we need to prove a technical estimate involving the constant
\begin{equation}
    \tau\coloneqq B^2 \|G^{-1}\|^2 \lambda.
\end{equation}

\begin{lemma}
\label{lem-thm-samp-grip-tauest} 
Assume that \eqref{eq-thm-samp-grip-assmodified} holds. Then 
\begin{equation}
    \|X_t x\|_{\cHt} \leq \sqrt{\tau},\quad t\in\cD,\ x\in\cB.
\end{equation}
\end{lemma}
\begin{proof}
Let $t\in\cD$, $x\in\cB$. Let $S\coloneqq\supp(x)$. Using \eqref{eq-thm-samp-grip-assmodified}, we obtain
\begin{align*}
    \|F_t\Phi^* x\|_{\cHt} &\leq
    \sum_{i\in S} |x_i| \|F_t\Phi^* e_i\|_{\cHt} \leq
    \sum_{i\in S}|x_i| B\omega_i \\ &\leq
    \|x\|_2 B\sqrt{\omega(S)} \leq \|G^{-1}\| \|Gx\|_2 B\sqrt{\lambda} \leq \sqrt{\tau}. \qedhere
\end{align*} 
\end{proof}

For $1<p<\infty$, we define the following seminorm:
\begin{equation}
    \|u\|_{X,p} \coloneqq \lc \sum_{k=1}^m \|X_{t_k}u\|_{\cHt}^{2p} \rc^{1/2p},\quad u\in\bC^M.
\end{equation}
We denote the corresponding pseudo-distance by $d_{X,p}$.
\begin{lemma}
\label{lem-thm-samp-grip-distest}
Let $1<r<\infty$ and $r'$ such that $\frac{1}{r}+\frac{1}{r'}=1$. Then we have
\begin{equation}
    d(u,v) \leq 2\tau^{(r-1)/2r}m^{1/2r}\|A^*A\|_{G,\omega,\lambda}^{1/2r} d_{X,r'}(u,v),\qquad u,v\in\cB.
\end{equation}
\end{lemma}
\begin{proof}
By H\"older's inequality, we have that
\begin{align*}
    d(u,v) &= \lc \sum_{k=1}^m \lc\|X_{t_k}u\|_{\cHt}^2-\|X_{t_k}v\|_{\cHt}^2\rc^2 \rc^{1/2} \\ &=
    \lc \sum_{k=1}^m \lc\|X_{t_k}u\|_{\cHt}+\|X_{t_k}v\|_{\cHt}\rc^2\lc\|X_{t_k}u\|_{\cHt}-\|X_{t_k}v\|_{\cHt}\rc^2 \rc^{1/2} \\ &\leq
    \lc \sum_{k=1}^m (\|X_{t_k}u\|_{\cHt}+\|X_{t_k}v\|_{\cHt})^{2r} \rc^{1/2r}
    \lc \sum_{k=1}^m (\|X_{t_k}u\|_{\cHt}-\|X_{t_k}v\|_{\cHt})^{2r'} \rc^{1/2r'} \\ &\leq
    2\sup_{w\in\cB} \lc \sum_{k=1}^m\|X_{t_k}w\|_{\cHt}^{2r} \rc^{1/2r}
    \lc \sum_{k=1}^m (\|X_{t_k}(u-v)\|_{\cHt})^{2r'} \rc^{1/2r'} \\ &=
    2\sup_{w\in\cB} \lc \sum_{k=1}^m\|X_{t_k}w\|_{\cHt}^{2r} \rc^{1/2r} d_{X,r'}(u,v). 
\end{align*}
By Lemma \ref{lem-thm-samp-grip-tauest} and \eqref{lem-thm-samp-grip-astara}, we get
\begin{align*}
    \sup_{w\in\cB} \lc \sum_{k=1}^m\|X_{t_k}w\|_{\cHt}^{2r} \rc^{1/2r} &=
    \sup_{w\in\cB} \lc \sum_{k=1}^m \|X_{t_k}w\|_{\cHt}^2\|X_{t_k}w\|_{\cHt}^{2(r-1)} \rc^{1/2r} \\ &\leq
    \tau^{(r-1)/2r} \sup_{w\in\cB} \lc \sum_{k=1}^m \|X_{t_k}w\|_{\cHt}^2 \rc^{1/2r} \\ &=
    \tau^{(r-1)/2r} m^{1/2r} \|A^*A\|^{1/2r}_{G,\omega,\lambda}. \qedhere
\end{align*}
\end{proof}

The relationship between $d(u,v)$ and $d_{X,r}(u,v)$ just proved reflects into an inequality for the corresponding covering numbers. 
\begin{proposition}
For any $r>1$, we have
\begin{equation}
    \int_0^{\infty} \sqrt{\log\cN( \cB,d,y)}\,\diff{y} \leq C(r) \int_0^{\infty} \sqrt{\log\cN( \cB,d_{X,r'},y)}\,\diff{y},
\end{equation}
where $C(r)\coloneqq 2\tau^{(r-1)/2r}m^{1/2r}\|A^*A\|_{G,\omega,\lambda}^{1/2r}$.
\end{proposition}
\begin{proof}
Lemma \ref{lem-thm-samp-grip-distest} implies that
\begin{equation}
    d(u,v)\leq C(r) d_{X,r'}(u,v),\quad u,v\in\cB.
\end{equation}
Standard properties of covering numbers (see, for instance, \cite{FR}) imply that
\begin{equation}
    \cN(\cB,d,y) \leq \cN(\cB,C(r) d_{X,r'},y) = \cN(\cB,d_{X,r'},y/C(r)),
\end{equation}
and therefore
\begin{equation}
    \int_0^{\infty} \sqrt{\log\cN(\cB,d,y)}\diff{y} \leq
    \int_0^{\infty} \sqrt{\log\cN(\cB,d_{X,r'},y/C(r))}\diff{y}.
\end{equation}
The substitution $y' = y/C(r)$ finally proves the claim.
\end{proof}

As a consequence of the modification of the metric, we accordingly change the underlying space. Consider the set
\begin{equation}
    T_{\omega}^{\lambda,M} \coloneqq \{x\in\bC^M\colon\ \|x\|_2\leq 1,\ \|x\|_{0,\omega}\leq \lambda\}.
\end{equation}
It is easy to prove that
\begin{equation}
    \cB \subseteq \|G^{-1}\| T_{\omega}^{\lambda,M}.
\end{equation}
\begin{proposition}
The following estimate holds:
\begin{equation}
    \int_0^{\infty} \sqrt{\log\cN(\cB,d_{X,r'},y)}\diff{y} \leq
    \|G^{-1}\| \int_0^{\infty} \sqrt{\log\cN(T_{\omega}^{\lambda,M},d_{X,r'},y)}\diff{y}.
\end{equation}
\end{proposition}
\begin{proof}
It suffices to apply the identity $\cN(\cB,Cd,y) = \cN(\cB,d,y/C)$ to $\cB\subseteq\|G^{-1}\| T_{\omega}^{\lambda,M}$ with respect to $d_{X,r'}$, then consider the substitution $y' = y/\|G^{-1}\|$.
\end{proof}

The efforts to modify the underlying structure of the covering numbers so far is repaid by the reduction to an easier setting, which in turn leads us to a convenient bound. In fact, the proof of the following result is essentially identical to that given in the proof of \cite[Theorem 5.2]{RW}, up to minimal adjustments, hence is omitted. 

\begin{proposition}
The following estimate holds:
\begin{equation}
    \int_0^{\infty} \sqrt{\log\cN(T_{\omega}^{\lambda,M},d_{X,r'},y)}\diff{y} \leq
    C_1\sqrt{r'm^{1/r'}\lambda\log^2{\lambda}\log{M}},
\end{equation}
where $C_1>0$ is a universal constant.
\end{proposition}

At this stage, we are ready to prove the desired bound for $\bE \delta_*$. 
\begin{proposition}
\label{prop-thm-samp-grip-expval-est}
The following estimate holds:
\begin{equation}
    \bE\delta_* \leq C_2 \sqrt{ \frac{\tau\log^3{\tau}\log{M}}{m} } \sqrt{\bE\delta_*+1},
\end{equation}
where $C_2>0$ is a universal constant.
\end{proposition}
\begin{proof}
Combining Proposition \ref{prop bound bE} with all the estimates proved so far, we ultimately get
\begin{equation}
    \bE\delta_* \leq C_2' m^{-1+1/2r+1/2r'} \tau^{(r-1)/2r} \|G^{-1}\| (r'\lambda\log^2{\lambda}\log{M})^{1/2}
    \bE\|A^*A\|_{G,\omega,\lambda}^{1/2r}, 
\end{equation} where $C_2'>0$ is a universal constant.
Note that $m^{-1+1/2r+1/2r'}=m^{-1/2}$. Given the arbitrariness in the choice of $r$, let us conveniently set
\begin{equation}
    r = 1+(\log{\tau})^{-1} \quad \Rightarrow \quad r'=1+\log{\tau}. 
\end{equation}
As a result, we obtain
\begin{equation}
    \tau^{(r-1)/2r} = \tau^{1/2r'} = \tau^{1/2(1+\log{\tau})} \leq e^{1/2},
\end{equation}
and, using that $\log\tau\ge\log 3\ge 1$,
\begin{equation}
    \sqrt{r'} = \sqrt{1+\log{\tau}} \leq \sqrt{2}\sqrt{\log{\tau}}.
\end{equation}
Finally, note that
\begin{equation}
    \|A^*A\|_{G,\omega,\lambda} \leq
    \|A^*A - G^*G\|_{G,\omega,\lambda}+\|G^*G\|_{G,\omega,\lambda}.
\end{equation}
In particular, for $x\in\cB$ we have
\begin{equation}
    |\langle G^*Gx,x \rangle| = \|Gx\|_2^2 \leq 1,
\end{equation}
hence $\|G^*G\|_{G,\omega,\lambda}\leq 1$. Therefore, 
\begin{align*}
    \|A^*A\|_{G,\omega,\lambda}^{1/2r} &\leq
    \lc \|A^*A-G^*G\|_{G,\omega,\lambda}+1 \rc^{1/2r} \\ &\leq
    \lc \|A^*A-G^*G\|_{G,\omega,\lambda}+1 \rc^{1/2}.
\end{align*}
By Jensen's inequality and \eqref{lem-thm-samp-grip-delta}, we conclude that
\begin{equation}
    \bE\|A^*A\|^{1/2r}_{G,\omega,\lambda} \leq
    \bE \lc \|A^*A-G^*G\|_{G,\omega,\lambda}+1 \rc^{1/2} \leq
    \sqrt{\bE\delta_*+1}. \qedhere
\end{equation}
\end{proof}

Given the bound just proved, it is easy to derive a lower bound for the number of samples $m$ needed in order to make $\bE \delta_*$ sufficiently small. 
\begin{proposition}
\label{prop-thm-samp-grip-expval-finest}
There exists a universal constant $C>0$ such that, for any $\delta\in(0,1)$,
\begin{equation}
    m \geq C\delta^{-2}\tau\log^3{\tau}\log{M}
    \quad\Rightarrow\quad \bE\delta_*\leq \delta/2.
\end{equation}
\end{proposition}
\begin{proof}
It is clear that for any $C\ge 16C_2^2$, the assumptions imply that
\begin{equation}
    C_2 \sqrt{ \frac{\tau\log^3{\tau}\log{M}}{m} } \leq \frac{\delta}{4}.
\end{equation}
By Proposition \ref{prop-thm-samp-grip-expval-est}, this implies that
\begin{equation}
\label{eq-prop-thm-samp-grip-expval-finest-01}
    \bE\delta_* \leq \frac{\delta}{4} \sqrt{\bE\delta_*+1}.
\end{equation}
A straightforward computation shows that, if $a\in(0,1)$ and $x>0$, then
\begin{equation}
    x \leq a\sqrt{x+1}
    \quad\Rightarrow\quad
    x \leq 2a.
\end{equation}
Therefore, from \eqref{eq-prop-thm-samp-grip-expval-finest-01} we infer
\begin{equation}
    \bE\delta_* \leq \delta/2,
\end{equation}
that is the claim. 
\end{proof}

The claim of Theorem \ref{thm-samp-grip} will be finally proved once a bound on the deviation of $\delta_*$ is established, in the spirit of the concentration inequality reported in Lemma \ref{lem-thm-samp-grip-bernstein}. To this aim, notice first that the map
\begin{equation}
    \bC^M\ni x \mapsto |\langle (A^*A-G^*G)x,x \rangle|\in\bR
\end{equation}
is continuous. Therefore, using $\cB'$ to denote a countable dense subset of $\cB$, we get 
\begin{equation}
    \delta_* = \sup_{x\in\cB'} |\langle(A^*A-G^*G)x,x\rangle|.
\end{equation}
Moreover, setting $\cP\coloneqq\{\pm 1\}$ allows us to write
\begin{equation}
    \delta_* = \sup_{(x,p)\in\cB'\times\cP} p\langle(A^*A-G^*G)x,x\rangle,
\end{equation}
and \eqref{lem-thm-samp-grip-astara} further yields 
\begin{equation}\label{eq:delta-star}
    \delta_* = \frac{1}{m} \sup_{(x,p)\in\cB'\times\cP} \sum_{k=1}^m f_{x,p}(X_{t_k}),
\end{equation}
where we set $f_{x,p}(B) \coloneqq p\langle (B^*B - G^*G)x,x \rangle$.

We are now in the position to obtain a tail bound for $\delta_*$. Indeed, the assumptions of the Bernstein-type inequality in Lemma \ref{lem-thm-samp-grip-bernstein} are satisfied, as shown below. 

\begin{lemma}
\label{lem-thm-samp-grip-bernstein-est}
For every $k=1,\dots,m$ and for $(x,p)\in\cB'\times\cP$,
\begin{equation}
    \bE f_{x,p}(X_{t_k}) = 0,\qquad
    |f_{x,p}(X_{t_k})|\leq \tau,\qquad
    \bE|f_{x,p}(X_{t_k})|^2 \leq \tau.
\end{equation}
\end{lemma}
\begin{proof}
The zero mean property of $f_{x,p}(X_{t_k})$ follows from  \eqref{lem-thm-samp-grip-xstarx}. 

Concerning the bound for $|\bE f_{x,p}(X_{t_k})|$, Lemma \ref{lem-thm-samp-grip-tauest} implies that 
\begin{equation}
    \|X_{t_k}x\|_{\cHt}^2 \leq \tau,\quad x\in\cB.
\end{equation}
On the other hand, we have that
\begin{equation}
    \|Gx\|_2^2 \leq 1\leq\tau,\quad x\in\cB.
\end{equation}
Therefore, for $(x,p)\in\cB'\times\cP$, we get that
\begin{equation}
\begin{split}
    |f_{x,p}(X_{t_k})| &\leq
    \max\lc \|X_{t_k}x\|_{\cHt}^2 - \|Gx\|_2^2, \|Gx\|_2^2 - \|X_{t_k}x\|_{\cHt}^2 \rc \\ &\leq
    \max\lc \|X_{t_k}x\|_{\cHt}^2 , \|Gx\|_2^2 \rc \leq \tau.
\end{split}
\end{equation}

Finally, we use the inequalities for $x\in\cB$ already proved, namely
\begin{equation}
    \|X_{t_k}x\|_{\cHt}^2\leq \tau;\quad
    \bE\|X_{t_k}x\|_{\cHt}^2=\|Gx\|_2^2 \leq 1.
\end{equation}
For $(x,p)\in\cB'\times\cP$, we thus get
\begin{align*}
    \bE|f_{x,p}(X_{t_k})|^2 &=
    \bE| \|X_{t_k}x\|_{\cHt}^2 - \|Gx\|_2^2 |^2 \\ &=
    \bE\|X_{t_k}x\|_{\cHt}^4 - \lc \bE\|X_{t_k}x\|_{\cHt}^2 \rc^2 \\ &\leq
    \tau\bE\|X_{t_k}x\|_{\cHt}^2 - \|Gx\|_2^4 \leq \tau. \qedhere
\end{align*} 
\end{proof}

Resorting to Lemma \ref{lem-thm-samp-grip-bernstein} gives the desired concentration bound. 
\begin{proposition}
\label{prop-thm-samp-grip-conc-ineq}
There exists a universal constant $C>0$ such that, for any $\delta,\gamma\in(0,1)$, if 
\begin{equation}
    m\geq C \delta^{-2} \tau\log(1/\gamma) \quad \text{and} \quad
    \bE\delta_*\leq \delta/2, 
\end{equation} then $\bP(\delta_*\geq\delta) \leq \gamma$. 
\end{proposition}
\begin{proof}
Since $\bE\delta_*\leq\delta/2$ by assumption, we have
\begin{equation}
    \bP(\delta_*\geq\delta) \leq
    \bP(\delta_*\geq\bE\delta_*+\delta/2) =
    \bP(m\delta_*\geq \bE(m\delta_*) + \varepsilon),
\end{equation}
where we set $\varepsilon = m\delta/2$. Recalling \eqref{eq:delta-star}, combining Lemma \ref{lem-thm-samp-grip-bernstein} (for $m\delta_*$) with the estimates from Lemma \ref{lem-thm-samp-grip-bernstein-est}, we get
\begin{align*}
    \bP(m\delta_*\geq \bE(m\delta_*) + \varepsilon) &\leq
    \exp\lc - \frac{\delta^2m^2/4}{2(m\tau+m\tau\delta)+m\tau \delta/3} \rc \\ &=
    \exp\lc - \frac{1}{4} \frac{1}{2(1+\delta)+\delta/3} \frac{\delta^2 m}{\tau} \rc \\ &\leq
    \exp\lc - \frac{1}{C} \frac{\delta^2 m}{\tau} \rc, 
\end{align*} where $C \ge 18$. The claim thus follows from the assumptions on $m$. 
\end{proof}

\subsection{The general case} 
\label{subsection-gencase}
We are ready to conclude the proof of Theorem \ref{thm-samp-grip} in full generality. So far, we have proved that if $f_{\nu}\equiv 1$, namely if $\mu=\nu$ is a probability measure, then by Proposition \ref{prop-thm-samp-grip-expval-finest} and Proposition \ref{prop-thm-samp-grip-conc-ineq} we get $\delta_*\leq \delta$ with probability exceeding $1-\gamma$. In particular, $A$ satisfies the g-RIP with respect to $(G,\omega,\delta,\lambda)$ with overwhelming probability. 

Consider now the case where $f_{\nu}\not\equiv 1$. We introduce the normalized measurement operators $\tilde{F}_t$ defined by
\begin{equation}
    \tilde{F}_t \coloneqq f_{\nu}(t)^{-1/2}F_t.
\end{equation}
Assumption \ref{ass-coherence-bound} now reads as follows:
\begin{equation}
    \|F_t\phi_i\|_{\cHt} \leq B\sqrt{f_{\nu}(t)}\omega_i,\quad t\in\cD,\ i\in\Gamma, 
\end{equation}
the latter being equivalent to 
\begin{equation}
    \|\tilde{F}_t\phi_i\|_{\cHt} \leq B\omega_i,\quad t\in\cD,\ i\in\Gamma.
\end{equation}
Moreover, note that $\tilde{F}_t$ satisfies the following identity with respect to $\diff{\nu} = f_{\nu}\diff{\mu}$ for every $u\in\cH_1$:
\begin{equation}
    \int_{\cD} \|\tilde{F}_t u\|_{\cHt}^2\diff{\nu}(t) =
    \int_{\cD} f_{\nu}(t)^{-1}\|F_t u\|_{\cHt}^2 f_{\nu}(t)\diff{\mu}(t) =
    \|Fu\|_{\cHd}^2.
\end{equation}
The proof thus reduces again to the case where $f_{\nu}\equiv 1$ for the measurement operators $\tilde{F_t}$ and the companion forward map $F$. In particular, with probability exceeding $1-\gamma$, we obtain that $\tilde{A}$ satisfies the g-RIP with respect to $(G,\omega,\delta,\lambda)$, where
\begin{equation}
    \tilde{A} \coloneqq \lc \frac{1}{\sqrt{m}} \tilde{F}_{t_k}\Phi^*\iota_\Lambda \rc_{k=1}^m =
    \lc \frac{1}{\sqrt{m}} f_{\nu}(t_k)^{-1/2}F_{t_k}\Phi^*\iota_\Lambda \rc_{k=1}^m = QA.
\end{equation}

\end{document}